\documentclass{article}%
\usepackage{amssymb,amsmath,amsthm,color}%
\providecommand{\U}[1]{\protect\rule{.1in}{.1in}}
\newcommand\Cov{{\mathrm {Cov}}}

\newcommand\bkE{{\mathbb {E}}}
\newcommand\E{{\mathbb {E}}}

\newcommand\beq{\begin{equation}}
\newcommand\eeq{\end{equation}}

\newtheorem{hypo}{Hypothesis}

\newtheorem{prop}[hypo]{Proposition}

\newtheorem{thm}[hypo]{Theorem}

\newtheorem{lem}[hypo]{Lemma}

\newtheorem{claim}[hypo]{Claim}

\newtheorem{defi}[hypo]{Definition}

\newtheorem{rqe}[hypo]{Remark}

\newtheorem{coro}[hypo]{Corollary}

\def\U{\mathcal{U}}

\begin{document}
\begin{center} {\bf \Large Empirical central limit theorems for ergodic automorphisms of the torus}\vskip15pt

J\'er\^ome Dedecker $^{a}$, Florence Merlev\`{e}de $^{b}$ and Fran\c{c}oise P\`ene  $^{c}$
\end{center}
\vskip10pt
$^a$ Universit\'e Paris Descartes, Sorbonne Paris Cit\'e, Laboratoire MAP5
and CNRS UMR 8145. E-mail: jerome.dedecker@parisdescartes.fr\\ \\
$^b$ Universit\'e Paris-Est, LAMA (UMR 8050), UPEMLV, CNRS, UPEC, F-77454
Marne-La-Vall\'ee. E-mail: florence.merlevede@univ-mlv.fr\\ \\
$^c$ Universit\'e de Brest,  Laboratoire de Math\'ematiques de Bretagne
Atlantique UMR CNRS 6205. Supported by the french ANR project Perturbations (ANR 10-BLAN 0106).
E-mail: francoise.pene@univ-brest.fr\vskip10pt
{\it Key words}:  Empirical distribution function, Kiefer process, Ergodic automorphisms of the torus,
Moment inequalities.\vskip5pt
{\it Mathematical Subject Classification} (2010): 60F17, 37D30.

\begin{center}
{\bf Abstract}\vskip10pt
\end{center}
Let $T$ be an ergodic automorphism of the $d$-dimensional torus ${\mathbb T}^d$, 
and $f$ be a continuous function from ${\mathbb T}^d$ to ${\mathbb R}^\ell$.
On the probability space 
${\mathbb T}^d$ equipped with the Lebesgue-Haar measure,
we prove the weak convergence of the sequential empirical process
of the sequence $(f \circ T^i)_{i \geq 1}$ under some mild conditions on the modulus of 
continuity of $f$. The proofs are based on new limit theorems, on new inequalities for 
non-adapted sequences, and on  new estimates of the conditional expectations of 
$f$ with respect to a natural filtration.

\section{Introduction}

Let $d\ge 2$ and $\mathbb T^d=\mathbb R^d/\mathbb Z^d$ be the
$d$-dimensional torus. For every $x\in\mathbb R^d$, we write $\bar x$ its class in $\mathbb T^d$. We denote by $\lambda$ the Lebesgue measure  on
$\mathbb R^d$, and by $\bar \lambda$ the Lebesgue measure on $\mathbb T^d$.

On the probability space $({\mathbb T}^d, \bar \lambda)$, we consider a group automorphism $T$ of $\mathbb T^d$.
We recall that $T$ is the quotient map of a linear map
$\tilde T:\mathbb R^d\rightarrow\mathbb R^d$
given by $\tilde T(x)=S\cdot x$, where $S$ is a $d\times d$-matrix with integer entries and with
determinant 1 or -1. The map $\tilde T$ preserves the infinite Lebesgue measure $\lambda$ on $\mathbb R^d$
and $T$ preserves the probability Lebesgue measure $\bar\lambda$.

We assume that $T$ is ergodic, which is equivalent to the fact that no eigenvalue of $S$
is a root of the unity. This hypothesis holds true in the case of hyperbolic automorphisms of
the torus (i.e. in the case when no  eigenvalue of $S$ has modulus one) but is much weaker.
Indeed, as mentionned in \cite{SLB}, the following matrix gives an example of an ergodic
non-hyperbolic automorphism of
$\mathbb T^4$~:
$$
S:=\left(\begin{array}{cccc}0&0&0&-1\\1&0&0&2\\0&1&0&0\\0&0&1&2\end{array}\right).
$$
When $T$ is ergodic but non-hyperbolic, the dynamical system $(\mathbb T^d,T,\bar\lambda)$
has no
Markov partition. However, it is possible to construct some measurable
partition (see \cite{Lind}), and to prove some decorrelation properties for
regular functions (see \cite{Lind,SLBFP}).

Let $\ell$ be some positive integer, and let  $f=(f_1, \ldots f_\ell)$ be a function from ${\mathbb T}^d$ to
${\mathbb R}^\ell$.
On the probability space $({\mathbb T}^d, \bar \lambda)$, the sequence
$(f\circ T^k)_{k\in {\mathbb Z}}$ is a stationary sequence of ${\mathbb R}^\ell$-valued random variables.
When $\ell=1$ and $f$ is square integrable,  Le Borgne \cite{SLB}
proved the functional central limit theorem and the Strassen strong invariance principle
for the partial sums
\begin{equation}\label{partial}
\sum_{i=1}^n (f \circ T^i - \bar \lambda (f))
\end{equation}
under weak hypotheses on the Fourier coefficients  of $f$, thanks to Gordin's method and to the partitions studied by Lind in \cite{Lind}. In the recent paper \cite{DeMePene}, we slightly improve
on Le Borgne's conditions,  and we  show how to obtain  rates of convergence  in the strong invariance principle
up to $n^{1/4}\log (n)$, by reinforcing the conditions on the Fourier coefficients of $f$.

Now,  for any $s \in {\mathbb R}^\ell$, define the partial sum
\begin{equation}\label{sn}
S_n(s)=  \sum_{k=1}^n ({\bf 1}_{f \circ T^k \leq s}-F(s)) \, ,
\end{equation}
where as usual ${\bf 1}_{f \circ T^k \leq s}=  {\bf 1}_{f_1 \circ T^k \leq s_1} \times \cdots
\times {\bf 1}_{f_\ell \circ T^k \leq s_\ell}$, and
$F(s)= \bar \lambda ({f  \leq s})$ is the multivariate
distribution function of $f$.

In this paper, we give some conditions on the modulus  of continuity of $f$
for the weak convergence
to a Gaussian process of the sequential empirical process
\begin{equation}\label{seqempirical}
\Big \{ \frac{S_{[nt]}(s)}{\sqrt n}, t \in [0,1], s \in {\mathbb R}^\ell \Big \} \, .
\end{equation}

The paper is organized as follows. Our main results are given in Section \ref{results} and proved in Section \ref{lastsec}.
The proofs require new probabilistic results established in Section \ref{probresults} combined with a key estimate for
toral automorphisms which is given in Section \ref{ineqauto}. Let us give now an overview of our results.

In Section \ref{Lp}, we consider the case where $\ell=1$ and $S_n$
is viewed as an ${\mathbb L}^p$-valued random variable for some $p\in [2, \infty[$ (this is possible because
$\int |S_n(s)|^p ds < \infty$ for any $p \in [2, \infty[$), so
that the sequential empirical process is an element of $D_{{\mathbb L}^p}([0,1])$, the space of
${\mathbb L}^p$-valued c\`adl\`ag functions. We prove the weak convergence 
of the process $\{ n^{-1/2} S_{[nt]}, t \in [0, 1]\}$ in
$D_{{\mathbb L}^p}([0,1])$ equipped with the uniform metric to a ${\mathbb L}^p$-valued
Wiener process, and we give the covariance operator of this Wiener process. The proof
is based on a new central limit theorem for dependent sequences with values in smooth Banach spaces, which is given in Section \ref{smooth}.

In Section \ref{secKiefer}, we state the convergence of the sequential empirical process (\ref{seqempirical}) in the space $\ell^\infty([0,1]\times {\mathbb R}^\ell)$
of bounded functions from $[0,1]\times {\mathbb R}^\ell$ to ${\mathbb R}$ equipped with the uniform metric. In that case, the limiting Gaussian process
is a generalization of the Kiefer process introduced by Kiefer in \cite{Kiefer} for the sequential empirical process of independent and identically distributed
random variables. The proof is based on a new Rosenthal inequality for dependent sequences (possibly non adapted), which is given in Section \ref{secRos}.
The weak convergence of the empirical process $\{ n^{-1/2}S_{n}(s),  s \in {\mathbb R}^\ell \}$
 has also been treated in \cite{DuJo}
and \cite{DeDu}. We shall be more precise on these two papers in Section \ref{secKiefer}.

To prove these results, we shall use a control of the conditional expectations
of continuous observables with respect to the filtration introduced by Lind
\cite{Lind}, involving the modulus of continuity of the observables (See Theorem
\ref{THM2}
of Section \ref{ineqauto}). As far as we know, such  controls were known for H\"older observables only
(see \cite{SLBFP}). Let us indicate that the inequalities given in Theorem \ref{THM2} are interesting by themselves. For instance one
can use them to establish weak invariance principle and
rates of convergence in the strong invariance principle
for the partial sums  (\ref{partial})
(see Section \ref{adresults}).

In this paper, the conditions on a function $f$ from  ${\mathbb T}^d$ to ${\mathbb R}$ will be expressed in terms of  its modulus of continuity $\omega(f,\cdot)$ defined as follows:
\begin{equation}\label{modulus}
\text{for $\delta>0$,}\ \ \ 
\omega(f,\delta):=\sup_{\bar x,\bar y\in\mathbb T^d\, :\, d_1(\bar x,\bar y)
   \le\delta}|f(\bar x)-f(\bar y)| \, ,
\end{equation}
where $d_1(\bar x,\bar y)=\min_{k\in\mathbb Z^d}\Vert x-y+k\Vert$ 
for some norm $\Vert \cdot \Vert$ on $\mathbb R^d$.

\section{Empirical central limit theorems}\label{results}
\setcounter{equation}{0}

\subsection{Empirical central limit theorem in ${\mathbb L}^p$}\label{Lp}

In this section, ${\mathbb L}^p$  is the space of Borel-measurable functions $g$ from $\mathbb R$ to $\mathbb R$ such that $\lambda (|g|^p) < \infty$,
$\lambda$ being the Lebesgue measure on ${\mathbb R}$. If $f$ is a bounded function, then, for any $p \in [2, \infty[$, the random variable  $S_n$ defined in (\ref{sn}) is an
${\mathbb L}^p$-valued random variable, and the process $\{ n^{-1/2} S_{[nt]}, t \in [0, 1]\}$ is a random
variable with values in $D_{{\mathbb L}^p}([0,1])$, the space of
${\mathbb L}^p$-valued c\`adl\`ag functions. In the next theorem, we give a condition on the modulus of continuity $\omega(f, \cdot)$ of $f$
under which  the process $\{ n^{-1/2} S_{[nt]}, t \in [0, 1]\}$ converges in distribution to an ${\mathbb L}^p$-valued  Wiener process,
in the space $D_{{\mathbb L}^p}([0,1])$ equipped with the uniform metric.
By an ${\mathbb L}^p$-valued  Wiener process with covariance operator $\Lambda_p$,
we mean a centered  Gaussian process
$W=\{W_t, t \in [0,1]\}$ such that
${\mathbb E}(\|W_{t}\|^2_{{\mathbb L}^p})<\infty$ for all
$t \in [0,1]$ and, for any $g,h$ in ${\mathbb L}^q$ ($q$ being the conjugate exponent of $p$),
$$
  {\mathrm{Cov}}\Big ( \int_{\mathbb R} g(u) W_{t}(u) du \, ,
  \int_{\mathbb R} h(u)  W_{s}(u) du \Big)=
  \min(t,s) \Lambda_p(g,h) \, .
$$

\begin{thm} \label{EmpLp}
Let $f: {\mathbb T}^d \rightarrow {\mathbb R}$ be a continuous function, with modulus
of continuity $\omega(f, \cdot)$. Let  $p \in [2, \infty[$, and let $q$ be its conjugate exponent.
  Assume that
  $$
 \int_0^{1/2} \frac{\big ( \omega(f,t) \big )^{1/p}}{t |\log t|^{1/p}} dt < \infty \, .
  $$
  Then the process $\{ n^{-1/2} S_{[nt]}, t \in [0, 1]\}$ converges
  in distribution in the space $D_{{\mathbb L}^p}([0,1])$ to an ${\mathbb L}^p$-valued
  Wiener process $W$, with covariance operator $\Lambda_p$  defined by
\begin{equation}\label{covdef}
   \Lambda_p (g,h) =\sum_{k \in {\mathbb Z}}
   \Cov\Big( \int_{\mathbb R} g(s) {\bf 1}_{f \leq s}ds, \int_{\mathbb R} h(s) {\bf 1}_{f\circ T^k \leq s} ds\Big)
  \, , \quad \text{for any $g,h$ in ${\mathbb L}^q$.}
\end{equation}
\end{thm}
The proof of Theorem \ref{EmpLp} is based on results of Sections \ref{probresults} and \ref{ineqauto} and is
postponed to Section \ref{lastsec}.
\begin{rqe}
In particular, if $f$ is H\"older continuous, then the conclusion of
Theorem \ref{EmpLp} holds for any $p \in [2, \infty [$.
\end{rqe}

Let us give an application of this theorem to the Kantorovich-Rubinstein distance between the empirical measure of $(f\circ T^i)_{1\leq i \leq n}$ and the distribution $\mu$ of $f$. Let
$$
  \mu_n= \frac{1}{n}\sum_{i=1}^n \delta_{f\circ T^i} \quad \text{and} \quad
  \mu_{n,k}= \frac{1}{n}\Big((n-k)\mu + \sum_{i=1}^k \delta_{f\circ T^i}\Big) \, .
$$
The Kantorovich distance between two probability measures $\nu_1$ and $\nu_2$
is defined as
$$
  K(\nu_1,\nu_2)= \inf \Big \{ \int|x-y| \nu(dx,dy),
  \nu \in {\mathcal M}(\nu_1,\nu_2) \Big \},
$$
where ${\mathcal M}(\nu_1,\nu_2)$ is the set of probability measures with margins
$\nu_1$ and $\nu_2$.
\begin{coro}\label{Kanto}
Let $f: {\mathbb T}^d \rightarrow {\mathbb R}$ be a continuous function, with modulus
of continuity $\omega(f, \cdot)$.
  Assume that
  $$
 \int_0^{1/2} \frac{\sqrt{\omega(f,t)}}{t \sqrt{|\log t|}} dt < \infty \, .
  $$
  Then $\sqrt n K(\mu_n, \mu)$ converges in distribution to
  $\|W_1\|_{{\mathbb L}^1}$, and
  $\sup_{1\leq k \leq n} \sqrt n K(\mu_{n,k}, \mu)$ converges in distribution to
  $\sup_{t \in [0,1]} \|W_t\|_{{\mathbb L}^1}$, where $W$ is the ${\mathbb L}^2$-valued
  Wiener process with covariance operator $\Lambda_2$ defined by $(\ref{covdef})$.
\end{coro}

\begin{proof}[Proof of Corollary \ref{Kanto}]
Applying Theorem \ref{EmpLp} with $p=2$, we know that
$\{ n^{-1/2} S_{[nt]}, t \in [0, 1]\}$ converges
in distribution in the space $D_{{\mathbb L}^2}([0,1])$ to an ${\mathbb L}^2$-valued
Wiener process $W$, with covariance operator $\Lambda_2$  defined by
(\ref{covdef}).
Since $f$ is continuous on ${\mathbb T}^d$,
it follows that $|f|\leq M$ for some positive constant $M$, so that $S_{[nt]}(s)=0$ and
$W_t(s)=0$ for any $t\in [0,1]$ and any $|s|>M$.
Since  $\|\cdot \|_{{\mathbb L}^1}$ is a continuous function on the space of functions
in ${\mathbb L}^2$ with support in $[-M, M]$, it follows
that $n^{-1/2} \|S_n\|_{{\mathbb L}_1}$ converges in distribution to
$\|W_1\|_{{\mathbb L}^1}$, and that $
\sup_{t \in [0,1]}n^{-1/2} \|S_{[nt]}\|_{{\mathbb L}_1}$ converges in distribution
to $\sup_{t \in [0,1]} \|W_t\|_{{\mathbb L}^1}$. Now, if $\nu_1$
and $\nu_2$ are probability measures on the real line, with distribution functions
$F_{\nu_1}$ and $F_{\nu_2}$ respectively,
$$
  K(\nu_1,\nu_2)= \int_{\mathbb R} |F_{\nu_1}(t)-F_{\nu_2}(t)| dt \, .
$$
Hence $nK(\mu_n, \mu)=\|S_n\|_{{\mathbb L}_1}$ and
$\sup_{1\leq k \leq n} nK(\mu_{n,k}, \mu)=\sup_{t \in [0,1]}\|S_{[nt]}\|_{{\mathbb L}_1}$,
and the result follows.
\end{proof}

\subsection{Weak convergence to the Kiefer process} \label{secKiefer}
Let $\ell$ be a positive integer.
Let $f=(f_1, \ldots, f_\ell)$ be a continuous function from ${\mathbb T}^d$ to ${\mathbb R}^\ell$. The modulus of continuity $\omega(f, \cdot) $ of $f$
is defined by
$$
  \omega(f, x)= \sup_{1\leq i \leq \ell} \omega(f_i,x) \, ,
$$
where we recall that $\omega(f_i,x) $ is defined by equation (\ref{modulus}).

As usual, we denote by $\ell^\infty([0,1]\times {\mathbb R}^\ell)$ the space of bounded functions from $[0,1]\times {\mathbb R}^\ell$ to ${\mathbb R}$ equipped with the
uniform norm. For details on weak convergence on the non separable space
$\ell^\infty([0,1]\times {\mathbb R}^\ell)$, we refer to \cite{VW}
(in particular, we shall not discuss any measurability problems, which
can be handled by using the outer probability).

For any positive integer $\ell$ and any $\alpha \in ]0,1]$, let
\begin{equation}\label{defaellalpha}
 a(\ell, \alpha)= \min_{p\geq \max(\ell +2, 2 \ell)} k_{\ell, \alpha}(p), \
 \text{where} \
 k_{\ell, \alpha}(p)= \max \Big (\frac{p}{\alpha (p-2\ell)} , \frac{(p-1) ( 2 \alpha + p)} {p \alpha} \Big )  \, .
\end{equation}
Note that this minimum is reached 
at $p_1=\max(3, p_0)$, where  $p_0$ is the unique solution in $]2\ell, 4 \ell[$  of the equation 
\begin{equation}\label{debileq}
  \frac{p}{(p-2\ell)}=\frac{(p-1)(p+2\alpha)}{p} 
\end{equation}
(in particular, $p_1=p_0$ if $\ell>1$). 

\medskip

We are now in position to state the main result of this section.

\begin{thm}\label{Kiefer}
Let $f=(f_1, \ldots, f_\ell): {\mathbb T}^d \rightarrow {\mathbb R}^\ell$ be a continuous function, with modulus
of continuity $\omega(f, \cdot)$. Assume that the distribution functions of the $f_i$'s are H\"older
continuous of order $\alpha \in ]0, 1]$. If
$$\omega(f,x)\leq C |\log (x)|^{-a} \quad \text{for some} \quad  a> a(\ell, \alpha) \, , 
$$
then the process $\{ n^{-1/2} S_{[nt]}(s), t \in [0, 1], s \in {\mathbb R}^\ell\}$ converges
in distribution in the space $\ell^\infty([0,1]\times {\mathbb R}^\ell)$ to a Gaussian process $K$
with covariance function $\Gamma$ defined by: for any $(t,t') \in [0,1]^2$ and
any $(s,s') \in {\mathbb R}^\ell \times {\mathbb R}^\ell$,
$$
\Gamma(t,t',s,s')=min(t,t') \Lambda (s, s') \quad \text{with} \quad
\Lambda(s,s')=\sum_{k \in {\mathbb Z}}{\mathrm{Cov}}({\bf 1}_{f \leq s}, {\bf 1}_{f\circ T^k \leq s'}) \, .
$$
\end{thm}
The proof of Theorem \ref{Kiefer} is given in Section \ref{lastsec}. It uses results of Sections \ref{probresults} and
\ref{ineqauto}.
\begin{rqe} \label{cardan}
Using the Cardan formulas (see the appendix) to solve (\ref{debileq}), we get 
$$p_0=2\frac{\ell+1-\alpha}3+2\sqrt{-\frac {p'} 3}\cos\left(\frac 13\arccos\left(-\frac q2\sqrt{\frac {27}{-(p')^3}}\right)\right) \, ,$$
with
$$p':= -4\alpha \ell+2 \ell-2\alpha-\frac 13 (-2\ell+2\alpha-2)^2<0 $$
and
$$q:=\frac 1{27}(-2\ell+2\alpha-2)(2(-2\ell+2\alpha-2)^2+36\alpha\ell-18\ell+18\alpha)+4\alpha\ell \, .$$
For example, for $\alpha=\ell=1$, we get $p_0\sim 2.9$
and finally $a(1,1)=10/3$.

Recall that, by Theorem \ref{EmpLp}, if $\ell=1$ and $p\in [2, \infty[$,  the weak invariance principle holds in $D_{{\mathbb L}^p}([0,1])$ as soon as $a>p-1$ 
without any condition on the distribution function of $f$.
\end{rqe}

The weak convergence of the (non sequential) empirical process $\{ n^{-1/2} S_{n}(s),  s \in {\mathbb R}^\ell\}$
has been studied in \cite{DuJo} and \cite{DeDu}. When $\ell=1$, a consequence of the main result of the paper \cite{DuJo}
is that the empirical process converges weakly to a Gaussian process for any H\"older continuous function $f$ having an H\"older
continuous distribution function. In the paper \cite{DeDu} this result is extended to any dimension $\ell$, under the assumptions that $f$ is H\"older continuous and that the moduli of continuity of the distribution functions of the $f_i$'s are smaller than $C|\log(x)|^{-a}$ in a neighborhood of $0$, for some $a>1$.

Note that, in our case,  one cannot apply Theorem 1 of \cite{DeDu}. Indeed, one cannot prove the multiple mixing for the sequence
$(f\circ T^i)_{i \in {\mathbb Z}}$ by assuming only that $\omega(f,x)\leq C |\log (x)|^{-a}$ in a neighborhood of zero (in that
case one can only prove that  $|{\mathrm {Cov}}(f, f\circ T^n)|$ is $O(n^{-a})$). However, even if our condition
on the regularity of $f$ is much weaker than in \cite{DeDu}, our result cannot be directly compared to that of \cite{DeDu},
because we assume that the distribution functions of the $f_i$'s are H\"older continuous of order $\alpha$, which is a stronger
assumption than the corresponding one in \cite{DeDu}.

\section{Probabilistic results}\label{probresults}
\setcounter{equation}{0}
In this section, $C$ is a positive constant which may vary from lines to lines, and
the notation
 $a_n \ll b_n$ means that there exists a numerical constant $C$ not
depending on $n$ such that  $a_n \leq  Cb_n$, for all positive integers $n$.

\subsection{Limit theorems and inequalities for stationary sequences}
Let $(\Omega,{\mathcal A}, {\mathbb P} )$ be a
probability space, and $T:\Omega \mapsto \Omega$ be
a bijective bimeasurable transformation preserving the probability ${\mathbb P}$.
For a $\sigma$-algebra ${\mathcal F}_0$ satisfying ${\mathcal F}_0
\subseteq T^{-1 }({\mathcal F}_0)$, we define the nondecreasing
filtration $({\mathcal F}_i)_{i \in {\mathbb Z}}$ by ${\mathcal F}_i =T^{-i
}({\mathcal F}_0)$. Let ${\mathcal {F}}_{-\infty} = \bigcap_{k \in {\mathbb
Z}} {\mathcal {F}}_{k}$ and ${\mathcal {F}}_{\infty} = \bigvee_{k \in
{\mathbb Z}} {\mathcal {F}}_{k}$.
Let ${\mathcal I}$ be the $\sigma$-algebra of $T$-invariant sets.
As usual, we say that $(T, {\mathbb P})$ is ergodic if each element $A$ of ${\mathcal I}$
is such that ${\mathbb P}(A)=0$ or 1.

Let $({\mathbb B}, |\cdot |_{\mathbb B})$ be a separable Banach space.
For a random variable $X$ with values in ${\mathbb B}$, let
$\Vert X \Vert_{p}= (  \E ( | X|^p _{{\mathbb B}}))^{1/p}$
and  ${\mathbb L}^{p} ({\mathbb B})$ be the space of ${\mathbb B}$-valued
random variables such that $\Vert X \Vert_{p}< \infty$.
For $X\in {\mathbb L}^1 ({\mathbb B})$, we shall use the notations $\E_k (X) = \E (X | {\mathcal F}_k)$, $\E_\infty (X) = \E (X | {\mathcal F}_\infty)$,
 $\E_{-\infty} (X) = \E (X | {\mathcal F}_{-\infty})$, and $P_k(X)= \E_{k}(X)-\E_{k-1}(X)$.
Recall that ${\mathbb E}(X|\mathcal F_n)\circ T^m={\mathbb E}(X\circ T^m|\mathcal F_{n+m})$.

Let $X_0$ be a random variable with values in ${\mathbb B}$.
Define the stationary sequence  $(X_i)_{i \in \mathbb Z}$ by
$X_i = X_0 \circ T^i$, and the partial sum $S_n$ by $S_n=X_1+X_2+\cdots + X_n$.

\subsubsection{Weak invariance principle in smooth Banach spaces}\label{smooth}
Following Pisier \cite{Pi}, we say that
a Banach space $({\mathbb B} , | \cdot |_{\mathbb B})$ is $2$-smooth if
there exists an equivalent norm $\|\cdot\|$ such that
$$
 \sup_{t >0} \Big \{\frac{1}{t^2} \sup \{\Vert x + ty  \Vert  + \Vert x
-ty  \Vert -2 :  \Vert x \Vert= \Vert y
\Vert = 1 \}\Big \}  < \infty \, .
$$
From \cite{Pi}, we know that if ${\mathbb B}$ is $2$-smooth and
separable, then there exists a constant $K$ such that, for any
sequence of  ${\mathbb B}$-valued martingale differences $(D_i)_{i
\geq 1}$,
\begin{equation}\label{assouad}
       {\mathbb E}(|D_1+ \cdots +D_n|_{\mathbb B}^2) \leq K
       \sum_{i=1}^n {\mathbb E}(|D_i|^2_{\mathbb B})\, .
\end{equation}
From \cite{Pi}, we see that $2$-smooth Banach spaces  play
the same role for martingales as spaces of type $2$  for sums of
independent variables.
Note that, for
any measure space $(T, {\mathcal A}, \nu)$, ${\mathbb L}^p(T, {\mathcal
A}, \nu)$ is $2$-smooth with $K=p-1$ for any $p \geq 2$, and that
any separable Hilbert space is $2$-smooth with $K=2$.

Let $D_{{\mathbb B}}([0,1])$ be the space of
${\mathbb B}$-valued c\`adl\`ag functions. In the next theorem, we give a condition
under which the process $\{ n^{-1/2} S_{[nt]}, t \in [0, 1]\}$ converges in distribution to a ${\mathbb B}$-valued  Wiener process,
in the space $D_{{\mathbb B}}([0,1])$ equipped with the uniform metric.

By a ${\mathbb B}$-valued  Wiener process with covariance operator $\Lambda_{\mathbb B}$,
we mean a centered  Gaussian process
$W=\{W_t, t \in [0,1]\}$ such that
${\mathbb E}(|W_{t}|^2_{{\mathbb B}})<\infty$ for all
$t \in [0,1]$ and, for any $g,h$ in
the dual space ${\mathbb B}^*$,
$$
  {\mathrm{Cov}}(  g (W_{t}),
   h(W_{s}) )=
  \min(t,s) \Lambda_{\mathbb B}(g,h) \, .
$$

\begin{prop}\label{PIFBanach}
Assume that ${\mathbb B}$ is a 2-smooth Banach space having a Schauder Basis,
that $(T, {\mathbb P})$ is ergodic, that
$\|X_0\|_{2}<\infty$ and that ${\mathbb E}(X_0)=0$.
If $\E_{- \infty} (X_0) = 0$ a.s., $X_0$ is ${\mathcal F}_{\infty}$-measurable, and 
\begin{equation}\label{condp0}
\sum_{k \in {\mathbb Z}} \|P_0(X_i)\|_{2} < \infty \, ,
\end{equation}
then the process
$
 \{ n^{-1/2} S_{[nt]} , t \in [0,1]  \}
$
converges in distribution in the space $D_{{\mathbb B}}([0,1])$ equipped with the uniform
metric
to a ${\mathbb B}$-valued
Wiener process $W_{\Lambda_{\mathbb B}}$,
where $\Lambda_{\mathbb B}$ is the covariance operator defined by
$$
\text{for any $g, h$ in ${\mathbb B}^*$,}  \quad
\Lambda_{\mathbb B}(g,h)= \sum_{k \in {\mathbb Z}} {\mathrm{Cov}}(g(X_0), h(X_k)) \, .
$$
\end{prop}

\begin{proof}[Proof of Proposition \ref{PIFBanach}]
Let us prove first that the result holds if ${\mathbb E}_{-1}(X_0)=0$
almost surely,
that is when $(X_k)_{k \in {\mathbb Z}}$ is a martingale difference sequence. As usual,
it suffices to prove that:
\begin{enumerate}
\item for any $0=t_0< t_1 < \cdots < t_d=1$
$$
  \frac{1}{\sqrt n}(S_{[nt_1]}, S_{[nt_2]}-S_{[nt_1]}, \cdots,
  S_{[nt_d]}-S_{[nt_{d-1}]})
$$
converges in distribution to  the Gaussian distribution $\mu$
on ${\mathbb B}^d$ defined by $\mu=\mu_{1} \otimes \mu_{2} \cdots \otimes
\mu_{d}$, where $\mu_{i}$ is the Gaussian distribution
on    ${\mathbb B}$    with
covariance operator $C_i$:
$$
\text{for any $g, h$ in ${\mathbb B}^*$,}  \quad
C_i(g,h)=  (t_i-t_{i-1}){\mathrm{Cov}}(g(X_0), h(X_0)) \, ;
$$
\item for any $\varepsilon>0$,
$$
\lim_{\delta \rightarrow 0} \limsup_{n \rightarrow \infty}
\frac 1 \delta
{\mathbb P}\Big(\max_{1 \leq k \leq [n\delta]} |S_k|_{\mathbb B} > \sqrt n
\varepsilon \Big)=0.
$$
\end{enumerate}
The first point can be proved exactly as in \cite{Wo}, who proved the result only
for $t_1=1$. Let us prove the second point. For any positive number $M$, let
$$
X_i'= X_i {\bf 1}_{|X_i|_{\mathbb B}\leq M} -
{\mathbb E}( X_i {\bf 1}_{|X_i|_{\mathbb B}\leq M}|{\mathcal F}_{i-1})
\quad \text{and} \quad X_i''= X_i-X_i'\, .
$$
Let also $S_{n}'ĩ= X_1'+ \cdots + X_n'$ and
$S_{n}''=X_1''+ \cdots + X_n''$. Since ${\mathbb B}$ is 2-smooth, Burkholder's inequality holds 
(see for instance \cite{Pinelis}), in such a way that
${\mathbb E}(\max_{1\le k\le n}|S_{k}'|_{\mathbb B}^q)\leq K_q M^q n^{q/2}$ for any $q\geq 2$.
Hence, applying Markov's inequality at order $q>2$,
$$
\frac 1 \delta {\mathbb P}\Big(\max_{1 \leq k \leq [n\delta]} |S_k'|_{\mathbb B} > \sqrt n
\varepsilon \Big) \leq
\frac{K_q M^q \delta^{(q-2)/2}}{\varepsilon^{q}}\, .
$$
As a consequence, we get 
\begin{equation}\label{tight1}
\lim_{\delta \rightarrow 0} \limsup_{n \rightarrow \infty}
\frac 1 \delta
{\mathbb P}\Big(\max_{1 \leq k \leq [n\delta]} |S'_k|_{\mathbb B} > \sqrt n
\varepsilon \Big)=0.
\end{equation}
In the same way, applying Markov's inequality at order $2$
\begin{equation}\label{tight2}
\frac 1 \delta
{\mathbb P}\Big(\max_{1 \leq k \leq [n\delta]} |S_k''|_{\mathbb B} > \sqrt n
\varepsilon \Big) \leq
\frac{K_2}{\varepsilon^2}{\mathbb E}(|X_0|_{\mathbb B}^2
{\bf 1}_{|X_0|_{\mathbb B}>M}) \, .
\end{equation}
The term ${\mathbb E}(|X_0|_{\mathbb B}^2
{\bf 1}_{|X_0|_{\mathbb B}>M})$ is as small as we wish by choosing $M$ large enough.
The point 2 follows from (\ref{tight1}) and (\ref{tight2}).

We now consider the general case. Since ${\mathbb B}$ is 2-smooth, Burkholder's
inequality holds and so Proposition 3.1
in \cite{DeMePene} (with $|\cdot|_{\mathbb B}$ instead of $|\cdot|_{\mathbb H}$)
applies: if (\ref{condp0}) holds, then, setting
$d_k= \sum_{i \in {\mathbb Z}} P_k(X_i)$, we have
\begin{equation}\label{ecart}
\Big \| \max_{1 \leq k \leq n } \Big  | \sum_{i=1}^k X_i - \sum_{i=1}^k d_i
\Big |_{\mathbb B} \Big \|_2 =o(\sqrt n).
\end{equation}
Since $(d_i)_{i \in {\mathbb Z}}$ is a stationary martingale differences sequence in ${\mathbb L}^2 ( {\mathbb B})$, we have 
just proved that it satisfies the conclusion of Proposition \ref{PIFBanach}.
From (\ref{ecart}) it follows that the conclusion of Proposition \ref{PIFBanach}
is also true for $(X_i)_{i \in {\mathbb Z}}$ with
$$
  \Lambda_{\mathbb B}(g,h)= {\mathrm {Cov}}(g(d_0),h(d_0)), \quad
  \text{for any $g, h$ in ${\mathbb B}^*$.}
$$

It remains to see that this covariance function can also be written as
in Proposition \ref{PIFBanach}.
Recall that since $\E_{- \infty} (X_0) = 0$ a.s. and  $X_0$ is ${\mathcal F}_{\infty}$-measurable, for any $g$ and $h$ in ${\mathbb B}^*$,                          
$$
\sum_{k \in {\mathbb Z}} |{\mathrm {Cov}}(g(X_0), h(X_k))| \leq
 \Big(\sum_{k \in {\mathbb Z}} \|P_0(g(X_k))\|_2 \Big) \Big(
 \sum_{k \in {\mathbb Z}} \|P_0(h(X_k))\|_2\Big)< \infty 
$$
(see the proof  of Theorem 3.1 in \cite{DeMePene}).
Hence, for any $g$ in ${\mathbb B}^*$,
\begin{equation}\label{covlim1}
\lim_{n \rightarrow \infty}
\frac 1 n {\mathbb E}\Big(\Big(\sum_{k=1}^n g(X_k)\Big)^2\Big)= \sum_{k \in {\mathbb Z}}
{\mathrm {Cov}}(g(X_0),g(X_k))\, .
\end{equation}
Now, from (\ref{ecart}), we also know that
\begin{equation}\label{covlim2}
\lim_{n \rightarrow \infty}
\frac 1 n {\mathbb E}\Big(\Big(\sum_{k=1}^n g(X_k)\Big)^2\Big)=
{\mathbb E}((g(d_0))^2)\, .
\end{equation}
Applying
(\ref{covlim1})
 and (\ref{covlim2}) with $g$, $h$ and $g+h$, we infer that
$$
{\mathrm {Cov}}(g(d_0),h(d_0))=\sum_{k \in {\mathbb Z}} {\mathrm{Cov}}(g(X_0), h(X_k))\, ,
$$
which completes the proof.
\end{proof}

\subsubsection{A Rosenthal inequality for non adapted sequences} \label{secRos}
We begin with a
maximal inequality that is useful to compare the moment of
order $p$ of the maximum of the partial sums of a non necessarily adapted process to the corresponding moment of the partial sum. The adapted version of this inequality  has been proven in the adapted case (that is when $X_0$ is ${\mathcal F}_0$-measurable)
in  \cite{MP}. Notice that Proposition 2 of \cite{MP} is stated for real valued random variables, but it holds also for variables taking values in a separable Banach space $({\mathbb B}, | \cdot |_{\mathbb B} )$. 
\begin{prop} \label{maxinequality} Let $p >1$ be a real number and $q$ be its conjugate exponent. Let $X_0$ be a random variable in ${\mathbb L}^p ({\mathbb B} ) $ and ${\mathcal F}_0 $ a $\sigma$-algebra satisfying ${\mathcal F}_0
\subseteq T^{-1 }({\mathcal F}_0)$. Then, for any integer $r$, the following inequality holds:
\begin{align} \label{inemaxna1}
\Big \Vert \max_{ 1 \leq m \leq 2^{r}} | S_m |_{\mathbb B}  \Big \Vert_p \leq q \Vert   S_{2^{r}%
}\Vert_p
 + q 2^{r/p} \sum_{\ell=0}^{r-1} 2^{-\ell/p}\Vert \bkE_0(S_{2^{\ell}})\Vert_{p}
 +  (q+1) 2^{r/p}\sum_{\ell=0}^{r} 2^{-\ell/p} \Vert S_{2^{\ell}} - {\mathbb{E}}_{2^{\ell}}(S_{2^{\ell}}%
) \Vert_{p}    \, .
\end{align}
\end{prop}
\begin{rqe} If we do not assume stationarity, so if we consider a sequence $(X_i)_{i \in {\mathbb Z}}$ in ${\mathbb L}^p ({\mathbb B} )$ for some $p>1$, and an increasing filtration $({\mathcal{F}}_{i})_{i \in {\mathbb Z}}$, our proof reveals that the following inequality holds true: for any integer $r$,
\begin{multline*}
\Big \Vert \max_{ 1 \leq m \leq 2^{r}} | S_m |_{\mathbb B}  \Big \Vert_p \leq q\Vert
S_{2^{r}}\Vert_{p}+q\sum_{l=0}^{r-1}\Big(\sum_{k=1}^{2^{r-l}-1}\Vert \E_{k2^{l}}
(S_{(k+1)2^{l}}-S_{k2^{l}} )\Vert_{p}^{p}\Big )^{1/p} \\
+(q +1) \sum_{l=0}^{r}\Big(\sum_{k=1}^{2^{r-l}}\Vert S_{k2^{l}}-S_{(k-1)2^{l}}-
 \E_{k2^{l}}(S_{k2^{l}}-S_{(k-1)2^{l}})\Vert_{p}^{p}\Big )^{1/p}\,.
\end{multline*}
\end{rqe}
\begin{rqe} \label{remarkinemax} Under the assumptions of Proposition \ref{maxinequality}, we also have that for any integer $n$,
\begin{align} \label{inemaxna2}
\Big \Vert \max_{ 1 \leq k \leq n} | S_k |_{\mathbb B}   \Big \Vert_p \leq 2 q
\max_{ 1 \leq k \leq n} \Vert S_k\Vert_p  +a_p n^{1/p}
\sum_{\ell = 1}^{n} \frac{\Vert{\mathbb E}_0 (S_{\ell} )  \Vert_p}{\ell^{1+1/p}} +
 b_pn^{1/p}\sum_{\ell = 1}^{2n} \frac{\Vert S_{\ell} - {\mathbb E}_{\ell} (S_{\ell} )  \Vert_p}{\ell^{1+1/p}}\, ,
\end{align}
where $$a_p = \frac{2^{1+ 1/p} q}{ 1 - 2^{-1-1/p} } \ \text{ and }
\ b_p= 2(q+1)\frac{2^{1+ 1/p}}
  { 1 - 2^{-1-1/p} }\, .$$
The proof of this remark will be done at the end of this section.
\end{rqe}

\bigskip

In the next results, we consider the case where $({\mathbb B}, |\cdot|_{\mathbb B})=
({\mathbb R}, |\cdot |)$. The next inequality is the non adapted version of the Rosenthal type inequality given in
\cite{MP} (see their Theorem 6).

\begin{thm}\label{Rosenthal1}
\label{directprop}Let $p >2$ be a real number and $q$ be its conjugate exponent.
Let $X_0$ be a real-valued random variable in ${\mathbb L}^p$ and ${\mathcal F}_0 $ a $\sigma$-algebra satisfying ${\mathcal F}_0
\subseteq T^{-1 }({\mathcal F}_0)$. 
Then, for any positive integer $r$, the following inequality
holds:
\begin{multline} \label{ineros}
{\mathbb{E}} \Big (\max_{1\leq j\leq 2^r}|S_{j}|^{p}\Big )\ll 2^r{\mathbb{E}
}(|X_{0}|)^{p}+2^r\left(  \sum_{k=0}^{r-1}\frac{\Vert{\mathbb{E}}
_{0}(S_{2^k})\Vert_{p}}{2^{k/p}}\right)^{p}+ 2^r\left(  \sum_{k=0}^{r}\frac{
      \Vert S_{2^k} - {\mathbb{E}}_{2^k}(S_{2^k})\Vert_{p}}{2^{k/p}}\right)^{p}   \\
 + 2^r\left(  \sum_{k=0}^{r-1} \frac{
     \Vert{\mathbb{E}}_{0}(S_{2^k}^{2})\Vert_{p/2}^{\delta}}{2^{2\delta k/ p}}\right)^{p/(2\delta)}\,,
\end{multline}
where $\delta=\min(1,1/(p-2))$.
\end{thm}
\begin{rqe} \label{remarknodiadic}
The inequality in the above theorem implies that for any positive integer $n$,
\begin{multline*}
{\mathbb{E}} \Big (\max_{1\leq j\leq n}|S_{j}|^{p}\Big )\ll n{\mathbb{E}%
}(|X_{1}|)^{p}+n\left(  \sum_{k=1}^{n}\frac{1}{k^{1+1/p\ }}\Vert{\mathbb{E}}%
_{0}(S_{k})\Vert_{p}\right)^{p}+ n\left(  \sum_{k=1}^{2n}\frac{1}{k^{1+1/p\ }}\Vert S_k - {\mathbb{E}}%
_{k}(S_{k})\Vert_{p}\right)^{p} \\
 + n\left(  \sum_{k=1}^{n}\frac{1}{k^{1+2\delta
/p}}\Vert{\mathbb{E}}_{0}(S_{k}^{2})\Vert_{p/2}^{\delta}\right)
^{p/(2\delta)}\, .
\end{multline*}
\end{rqe}
To prove Remark \ref{remarknodiadic}, it suffices to use the arguments developed in the proof of Remark \ref{remarkinemax} together with the following additional subadditivity property: for any integers $i$ and $j$, and any $\delta \in]0,1]$: $$\Vert{\mathbb{E}}%
_{0}(S_{i+j}^{2})\Vert^{\delta}_{p/2}\leq 2^{\delta} \Vert{\mathbb{E}}_{0}(S_{i}^{2})\Vert
_{p/2}+2^{\delta}\Vert{\mathbb{E}}_{0}(S_{j}^{2})\Vert_{p/2}\, .$$ 
So, according to the first item of Lemma 37 of \cite{MP}, for any integer $n \in ]2^{r-1}, 2^r]$,
$$
\sum_{k=0}^{r-1} \frac{
     \Vert{\mathbb{E}}_{0}(S_{2^k}^{2})\Vert_{p/2}^{\delta}}{2^{2\delta k/ p}} \ll \sum_{k=1}^{n}\frac{1}{k^{1+2\delta
/p}}\Vert{\mathbb{E}}_{0}(S_{k}^{2})\Vert_{p/2}^{\delta} \, .
$$
\begin{rqe} \label{banachvalued}  
Theorem \ref{Rosenthal1} has been stated in the real case. Notice that if we assume $X_0$ to be  in ${\mathbb L}^p ({\mathbb B} ) $ where $({\mathbb B}, |\cdot|_{\mathbb B})$ is a separable Banach space and $p$ is a real number in $]2, \infty[$, then a Rosenthal-type inequality similar as (\ref{ineros}) can be obtained but with a different $\delta$ for $2 < p < 4$. To be more precise, we get 
\begin{equation} \label{rosenthalremark}
{\mathbb{E}} \Big (\max_{1\leq j\leq 2^r}|S_{j}|_{\mathbb B}^{p}\Big )\ll 2^r{\mathbb{E}
}(|X_{0}|_{\mathbb B})^{p}+ 2^r\left(  \sum_{k=0}^{r}\frac{
      \Vert S_{2^k} - {\mathbb{E}}_{2^k}(S_{2^k})\Vert_{p}}{2^{k/p}}\right)^{p}   
 + 2^r\left(  \sum_{k=0}^{r-1} \frac{
     \Vert{\mathbb{E}}_{0}( | S_{2^k} |_{\mathbb B}^{2})\Vert_{p/2}^{\delta}}{2^{2\delta k/ p}}\right)^{p/(2\delta)}\,,
\end{equation}
where $\delta=\min(1/2,1/(p-2))$. The proof of this inequality is given 
at the end of this section.
\end{rqe}
As a consequence of (\ref{ineros}), one can prove  the following proposition which will be a key tool to prove the tightness
of the sequential empirical process (\ref{seqempirical}) in the space
$\ell^\infty([0,1]\times {\mathbb R}^\ell)$ (see the proof of Theorem \ref{Kiefer},
Section \ref{lastsec}). 
\begin{prop}
\label{consdirect} Let $p>2$. Let $X_0$ be a real-valued random variable in ${\mathbb L}^p$ and ${\mathcal F}_0 $ a $\sigma$-algebra satisfying ${\mathcal F}_0
\subseteq T^{-1 }({\mathcal F}_0)$. For any $j\geq1$, let
\begin{equation}
A(X,j)=\max \Big ( 2 \sup_{i\geq
0}\Vert{\mathbb{E}}_{0}(X_{i}X_{j+i})\Vert
_{p/2}, \sup_{0 \leq i \leq j
}\Vert{\mathbb{E}}_{0}(X_{j}X_{j+i})-{\mathbb{E}}(X_{j}X_{j+i})\Vert
_{p/2} \Big )\,.\label{notalambda}%
\end{equation}
Then, for every positive integer $n$,
\begin{multline*}
\Big \Vert \max_{1\leq j\leq n} |S_{j}| \Big \Vert_{p}\ll n^{1/2}%
\Big (\text{ }\sum_{k=0}^{n-1}|{\mathbb{E}}(X_{0}X_{k})|\Big )^{1/2}%
+n^{1/p} \Vert X_{1} \Vert_{p}
+ n^{1/p}\sum_{k=1}^{n}\frac{1}{k^{1/p}}\Vert {\mathbb{E}}_{0}(X_{k}%
)\Vert_{p} \\ + n^{1/p}\sum_{k=1}^{2n}\frac{1}{k^{1/p}}\Vert X_0 - {\mathbb{E}}_{k}(X_{0}%
)\Vert_{p} + n^{1/p}\Big (\sum_{k=1}^{n}\frac{1}{k^{(2/p)-1}}(\log k)^{\gamma}
A(X, k)\Big )^{1/2}\,.
\end{multline*}
where $\gamma$ can be taken $\gamma=0$ for $2<p\leq3$ and $\gamma>p-3$ for
$p>3$. The constant that is
implicitly involved in the notation $\ll$ depends on $p$ and $\gamma$ but it
depends neither on $n$ nor on the $X_{i}$'s.
\end{prop}
The proof of this proposition is left to the reader since it uses the same arguments as those developed for the proof of Proposition 20 in \cite{MP}.

\medskip    

We would like also to point out that Theorem \ref{directprop} implies the following Burkholder-type inequality. This has been already mentioned in the adapted case in \cite[Corollary 13]{MP}.

\begin{coro} \label{corburk}
Let $p>2$ be a real number, $X_0$ be a real-valued
 random variable in ${\mathbb L}^p$ and ${\mathcal F}_0 $ a $\sigma$-algebra satisfying ${\mathcal F}_0
\subseteq T^{-1 }({\mathcal F}_0)$. Then, for any integer $r$, the following inequality holds:
\[
{\mathbb{E}} \Big (\max_{1\leq j\leq 2^r}|S_{j}|^{p}\Big )\ll 2^{r p/2}{\mathbb{E}%
}(|X_{0}|^{p}) +2^{r p/2}\Big(\sum_{j=0}^{r-1}\frac{\Vert{\mathbb{E}}_{0}%
(S_{2^j})\Vert_{p}}{2^{j/2}}\Big)^{p} + 2^{r p/2}\Big (  \sum
_{j=1}^{r}\frac{\Vert S_{2^j} - {\mathbb{E}}%
_{2^j}(S_{2^j})\Vert_{p}}{2^{j/2}}\Big )
^{p}\,.
\]
\end{coro}
The above corollary (up to constants) is then the non adapted version of \cite[Theorem 1]{PUW} when $p>2$.

We now give the proof of the results of this section.

\medskip

\begin{proof}[Proof of Proposition \ref{maxinequality}] For any $k \in \{1, \dots, 2^{r}\}$, we have 
$$
S_{k} = S_k - \bkE_k(S_k ) + \bkE_k(S_{2^{r}} )- \bkE_k(S_{2^{r}} - S_k )  \, .
$$
Consequently
\begin{multline}
\Big \Vert \max_{ 1 \leq k \leq 2^{r}} | S_k |_{{\mathbb B}} \Big \Vert_p  \leq
\Big \Vert \max_{1\leq k\leq2^{r}}|{\mathbb{E}}_k(S_{2^{r}
})|_{{\mathbb B}} \Big \Vert_p + \Big \Vert \max_{1\leq m\leq2^{r}-1}|{\mathbb{E}}_{2^{r}-m}
(S_{2^{r}}-S_{2^{r}-m})|_{{\mathbb B}} \Big \Vert_p  \\
+ \Vert S_{2^r} - \bkE_{2^r}(S_{2^r} )\Vert_p +
\Big \Vert \max_{1\leq m\leq 2^{r}-1}|S_m - \bkE_m(S_m )|_{{\mathbb B}} \Big \Vert_p \,. \label{dec1max}
\end{multline}
Following the proof of Proposition 2 in \cite{MP}, we get 
\begin{multline*} 
\Big \Vert \max_{1\leq k\leq2^{r}}|{\mathbb{E}}_k(S_{2^{r}
})|_{{\mathbb B}} \Big \Vert_p + \Big \Vert \max_{1\leq m\leq2^{r}-1}|{\mathbb{E}}_{2^{r}-m}
(S_{2^{r}}-S_{2^{r}-m})|_{{\mathbb B}} \Big \Vert_p   \\
 \leq q \, \Vert {\mathbb{E}}_{2^r}(S_{2^{r}}) \Vert_{p} + q\sum_{\ell=0}^{r-1}\Big(\sum_{k=1}^{2^{r-\ell}-1}\Vert {\mathbb{E}%
}_{k2^{\ell}}(S_{(k+1)2^{\ell}}-S_{k2^{\ell}}) \Vert_{p}^{p}\Big )^{1/p}\ \, .
\end{multline*}
So, by stationarity,
\begin{multline} \label{momentb1max}
\Big \Vert \max_{1\leq k\leq2^{r}}|{\mathbb{E}}_k(S_{2^{r}
})|_{{\mathbb B}} \Big \Vert_p + \Big \Vert \max_{1\leq m\leq2^{r}-1}|{\mathbb{E}}_{2^{r}-m}
(S_{2^{r}}-S_{2^{r}-m})|_{{\mathbb B}} \Big \Vert_p  \\
 \leq  q \, \Vert {\mathbb{E}}_{2^r}(S_{2^{r}}) \Vert_{p}  + q 2^{r/p} \sum_{\ell=0}^{r-1} 2^{-\ell/p}
    \Vert {\mathbb{E}}(S_{2^{\ell}}
|{\mathcal{F}}_{0})\Vert_{p} \, .
\end{multline}
We  now bound the last term in the right hand side of (\ref{dec1max}). For any $m \in \{1, \dots, 2^{r}-1\}$, we consider its binary expansion:
\begin{equation*} \label{binary}
m=\sum_{i=0}^{r-1}b_{i}(m)2^{i},\ \text{ where $b_{i}(m)=0$ or $b_{i}(m)=1$%
}\,.
\end{equation*}
Set $m_{l}=\sum_{i=l}^{r-1}b_{i}(m)2^{i}$, and write that 
\beq \label{decdyadic}
|S_m - {\mathbb E}_m(S_m )|_{{\mathbb B}}
\leq\sum_{l=0}^{r-1}|S_{m_l} - S_{m_{l+1}} -  {\mathbb{E}}_m(S_{m_l} - S_{m_{l+1}})|_{{\mathbb B}}\, ,
\eeq
since $S_0 = 0$ and $m_r=0$. Now, since for any $l=0, \dots, r-1$, ${\mathcal{F}}_{m_l}\subseteq  {\mathcal{F}}_{m}$, the following decomposition holds:
\begin{multline*}
|S_{m_l} - S_{m_{l+1}} -  {\mathbb{E}}_m(S_{m_l} - S_{m_{l+1}})|_{{\mathbb B}} \leq |S_{m_l} - S_{m_{l+1}} -  {\mathbb{E}}_{m_l}(S_{m_l} - S_{m_{l+1}})|_{{\mathbb B}} \\
 +  \big |{\mathbb{E}} \big ( S_{m_l} - S_{m_{l+1}} -  {\mathbb{E}}_{m_l}(S_{m_l} - S_{m_{l+1}})|{\mathcal{F}}_{m}) \big ) \big |_{{\mathbb B}} \, .
\end{multline*}
Notice that $m_{l}\neq m_{l+1}$ only if $m_{l}%
=k_{m,l}2^{l}$ with $k_{m,l}$ odd. Then, setting
$$
B_{r,l} = \max_{1\leq k\leq2^{r-l},k\text{ odd}}|S_{k2^{l}}-S_{(k-1)2^{l}} - {\mathbb{E}}_{k2^{l}}(S_{k2^{l}}-S_{(k-1)2^{l}} )|_{{\mathbb B}} \, ,
$$
it follows that
$$
|S_{m_l} - S_{m_{l+1}} -  {\mathbb{E}}_m(S_{m_l} - S_{m_{l+1}})|_{{\mathbb B}} \leq  B_{r,l} +  |{\mathbb{E}}  ( B_{r,l}|{\mathcal{F}}_{m}) | \, .
$$
Starting from (\ref{decdyadic}), we then get
$$
\Big \Vert \max_{1\leq m\leq 2^{r}-1}|S_m - \bkE_m(S_m )|_{{\mathbb B}} \Big \Vert_p
    \leq \sum_{l=0}^{r-1} \Vert B_{r,l} \Vert_p + \sum_{l=0}^{r-1}
   \Big \Vert \max_{1\leq m\leq 2^{r}-1}|{\mathbb{E}}  ( B_{r,l}|{\mathcal{F}}_{m}) |
  \Big \Vert_p \, .
$$
Since $({\mathbb{E}}(B_{r,l}
|{\mathcal{F}}_{m}))_{m \geq1}$ is a
martingale, by using  Doob's maximal inequality, we get 
\[
\Big \Vert \max_{1\leq m\leq 2^{r}-1}|{\mathbb{E}}  ( B_{r,l}|{\mathcal{F}}_{m}) |
   \Big \Vert_p \leq q   \Vert {\mathbb{E}}(B_{r,l}|{\mathcal{F}
}_{2^{r}-1}) \Vert_p \leq  q\Vert B_{r,l}\Vert_{p}\, ,
\]
yielding to
$$
\Big \Vert \max_{1\leq m\leq 2^{r}-1}|S_m - {\mathbb E}_m(S_m )|_{\mathbb B} \Big \Vert_p
  \leq (q+1) \sum_{l=0}^{r-1} \Vert B_{r,l} \Vert_p \, .
$$
Since
$$
B_{r,l} \le \Biggl ( \sum_{k=1}^{2^{r-l}-1}|S_{k2^{l}}-S_{(k-1)2^{l}} -
    {\mathbb{E}}_{k2^{l}}(S_{k2^{l}}-S_{(k-1)2^{l}} )|_{\mathbb B}^p \Biggr  )^{1/p} \, ,
$$
we derive that
\begin{multline*}
\Big \Vert \max_{1\leq m\leq 2^{r}-1}|S_m - \bkE_m(S_m )|_{\mathbb B} \Big \Vert_p \\
\leq (q+1) \sum_{l=0}^{r-1}\Big(\sum_{k=1}^{2^{r-l}-1}\Vert S_{k2^{l}}-S_{(k-1)2^{l}}-
 \E_{k2^{l}}(S_{k2^{l}}-S_{(k-1)2^{l}})\Vert_{p}^{p}\Big )^{1/p} \, .
\end{multline*}
So, by stationarity, 
\begin{equation} \label{momentb2max}
\Big \Vert \max_{1\leq m\leq 2^{r}-1}|S_m - {\mathbb E}_m(S_m )|_{\mathbb B}
  \Big \Vert_p \leq (q+1) 2^{r/p}  \sum_{l=0}^{r-1} 2^{-l/p}  \Vert S_{2^{l}}- {\mathbb{E}}_{2^{l}}(S_{2^{l}} ) \Vert_p \, .
\end{equation}
Starting from (\ref{dec1max}) and taking into account (\ref{momentb1max})
and (\ref{momentb2max}), the inequality (\ref{inemaxna1}) follows.

\end{proof}

\begin{proof}[Proof of Theorem \ref{Rosenthal1}]
Thanks to Proposition \ref{maxinequality}, it suffices to prove that the inequality (\ref{ineros}) is satisfied  for ${\mathbb{E}} \big (|S_{2^r}|^{p}\big )$ instead of ${\mathbb{E}} \big (\max_{1\leq j\leq 2^r}|S_{j}|^{p}\big )$. We shall use similar dyadic induction arguments as those developed in the proof of Theorem 6 in \cite{MP}. With the notation  $a_n = \Vert S_n \Vert_p$, we shall establish the following recurrence formula: for any
positive integer $n$ and any $p >2$,
\begin{equation}
a_{2n}^{p}\leq 2a_{n}^{p}+c_{1}a_{n}^{p-1} \big ( \Vert{\mathbb{E}}_{0}(S_{n}%
)\Vert_{p} + \Vert S_n - {\mathbb{E}}%
_{n}(S_{n})\Vert_{p}\big )+c_{2}a_{n}^{p-2\delta}\Vert{\mathbb{E}}_{0}(S_{n}^{2})\Vert
_{p/2}^{\delta}\, , \label{recurrence}%
\end{equation}
where $c_{1}$ and $c_{2}$ are positive constants depending only on $p$. Before
proving it, let us show that (\ref{recurrence}) implies our result. With this aim, we give the following lemma which is a slight modification of Lemma 11 in \cite{MP}.

\begin{lem}
\label{reclemma}Assume that for some $0<\delta\leq1$ the recurrence formula
(\ref{recurrence}) holds. Then, for any integer $r$, 
\begin{multline}
a_{2^{r}}^{p}\leq 2^{r}\Big (4a_{2^{0}}^{p}+\Big (2c_{1}\sum_{k=0}%
^{r-1}2^{-k/p}\Vert{\mathbb{E}}_{0}(S_{2^{k}})\Vert_{p}\Big )^{p} \\+ \Big (2c_{1}\sum_{k=0}%
^{r-1}2^{-k/p}\Vert S_{2^k } - {\mathbb{E}}_{2^k}(S_{2^{k}})\Vert_{p}\Big )^{p}
+\Big (2c_{2}\sum_{k=0}^{r-1}2^{-2k\delta/p\ }\Vert{\mathbb{E}}_{0}(S_{2^{k}%
}^{2})\Vert_{p/2}^{\delta}\Big )^{p/2\delta}\Big )\,. \label{boundpropdyatic2}%
\end{multline}
\end{lem}
Let us prove the lemma. From inequality (\ref{recurrence}), by recurrence on
the first term, we obtain, for any positive integer $r$,
\begin{multline*}
a_{2^{r}}^{p}\leq 2^{r}\Big (a_{2^{0}}^{p}+ c_{1}\sum_{k=0}^{r-1}2^{-k-1}a_{2^{k}%
}^{p-1}\Vert{\mathbb{E}}_{0}(S_{2^{k}})\Vert_{p}+ c_1\sum_{k=0}^{r-1}2^{-k-1}a_{2^{k}%
}^{p-1}\Vert S_{2^{k} } - {\mathbb{E}}_{2^k}(S_{2^{k}})\Vert_{p} \\ + c_{2}\sum_{k=0}^{r-1}%
2^{-k-1}a_{2^{k}}^{p-2\delta}\Vert{\mathbb{E}}_{0}(S_{2^{k}}^{2})\Vert
_{p/2}^{\delta}\Big )\,. \label{recurrence2}%
\end{multline*}
With the notation $\displaystyle B_{r}=\max_{0\leq k\leq r}(a_{2^{k}}%
^{p}/2^{k}),$ it follows that
\begin{multline*}
B_{r}\leq a_{2^{0}}^{p}+c_{1}B_{r}^{1-1/p}\sum_{k=0}^{r-1}2^{-1-k/p}%
\Vert{\mathbb{E}}_{0}(S_{2^{k}})\Vert_{p}+c_{1}B_{r}^{1-1/p}\sum_{k=0}^{r-1}2^{-1-k/p}\Vert S_{2^{k} } - {\mathbb{E}}_{2^k}(S_{2^{k}})\Vert_{p} \\
+ c_{2}B_{r}^{1-2\delta/p}\sum
_{k=0}^{r-1}2^{-1-2k\delta/p\ }\Vert{\mathbb{E}}_{0}(S_{2^{k}}^{2})\Vert
_{p/2}^{\delta}\,.
\end{multline*}
Therefore, taking into account that either $B_{r}\leq 4 a_{2^{0}}^{p}$ or
$B_{r}^{1/p}\leq 4 c_{1}\sum_{k=0}^{r-1}2^{-1-k/p}\Vert{\mathbb{E}}_{0}(S_{2^{k}%
})\Vert_{p}$ or $B_{r}^{1/p}\leq 4 c_{1}\sum_{k=0}^{r-1}2^{-1-k/p}\Vert S_{2^{k} } - {\mathbb{E}}_{2^k}(S_{2^{k}})\Vert_{p}$ or $B_{r}^{2\delta/p}\leq 4c_{2}\sum_{k=0}^{r-1}2^{-1-2k\delta
/p\ }\Vert{\mathbb{E}}_{0}(S_{2^{k}}^{2})\Vert_{p/2}^{\delta}$, the inequality \eqref{boundpropdyatic2} follows. 

\smallskip

To end the proof of Theorem \ref{Rosenthal1}, it remains to prove \eqref{recurrence}. With this aim,  we denote by  $\bar{S}_{n}=X_{n+1}+\dots+X_{2n}$, and we write 
$$
S_{2n} = S_n -\E_n(S_n) + \E_n(S_n) + \bar{S}_{n} \,  .
$$
Recall now the following algebraic inequality: Let $x$ and $y$ be two positive real numbers and $p\geq1$ any real
number. Then
\begin{equation}
(x+y)^{p}\leq x^{p}+y^{p}+4^{p}(x^{p-1}y+xy^{p-1})  \label{comparaison}%
\end{equation}
(see Inequality (87) in \cite{MP}). The above inequality with $x=|\E_n(S_n) + \bar{S}_{n}|$ and $y= |S_n -\E_n(S_n)|$ gives
\begin{multline*}
a_{2n}^{p} \leq \Vert \E_n(S_n) + \bar{S}_{n} \Vert_p^p  + \Vert  S_n -\E_n(S_n) \Vert_p^p \\
 + 4^p   \E \big ( |\E_n(S_n) + \bar{S}_{n}|^{p-1} \times |S_n -\E_n(S_n)| \big )+4^p   \E \big ( |\E_n(S_n) + \bar{S}_{n}| \times  |S_n -\E_n(S_n)|^{p-1} \big ) \, .
\end{multline*}
Next using H\"older's inequality and stationarity, we derive that, for any $p \geq 2$,
\begin{equation}
 a_{2n}^{p} \leq \Vert \E_n(S_n) + \bar{S}_{n} \Vert_p^p
 + 2^{p-1}(1+2^{2p+1}) a_n^{p-1} \Vert S_n - \E_n(S_n) \Vert_p \, .
\label{recurrence2}
\end{equation}
Starting from (\ref{recurrence2}), (\ref{recurrence}) will follow if we can prove that  there exist two positive constants $c$ and $c_{2}$  depending only on $p$ such that
\begin{equation}
\Vert \E_n(S_n) + \bar{S}_{n} \Vert_p^p  \leq 2a_{n}^{p}+ c \, a_{n}^{p-1} \Vert{\mathbb{E}}_{0}(S_{n}%
)\Vert_{p} +c_{2}a_{n}^{p-2\delta}\Vert{\mathbb{E}}_{0}(S_{n}^{2})\Vert
_{p/2}^{\delta}\,.
\label{recurrence3}
\end{equation}
This inequality can be proven by following the lines of the end of the proof of Theorem 6 in \cite{MP} replacing in their proof $x=S_n$ by $x=\E_n (S_n)$. However, for reader's convenience we shall give the details. The proof is divided in three cases according to the values of $p$. 

Assume first that $2 < p \leq 3$. Inequality (85) in \cite{MP} applied with
$x=\E_n (S_n)$ and $y=\bar S_n$, gives
\[
|\E_n(S_n) + \bar{S}_{n}|^{p}\leq |\E_n(S_n)|^p + |\bar{S}_{n}|^{p} +p|\E_n (S_n)|^{p-1}\mathrm{sign}(\E_n (S_n))\bar{S}_{n}+ \frac{p(p-1)}{2}|\E_n(S_{n})|^{p-2}\bar{S}_{n}^{2}\, .
\]
But $ \E \big ( |\E_n(S_n)|^p \big ) \leq a_n^p$ and, by stationarity, $\E \big ( |\bar{S}_{n}|^{p} \big ) = a_n^p$. Moreover, H\"older's inequality combined with stationarity gives
$$
\E \big (|\E_n (S_n)|^{p-1}\mathrm{sign}(\E_n (S_n))\bar{S}_{n} \big ) = \E \big (|\E_n (S_n)|^{p-1}\mathrm{sign}(\E_n (S_n))\E_n ( \bar{S}_{n}) \big ) \leq a_n^{p-1} \Vert \E_0 ( S_n) \Vert_p \, ,
$$
and
$$
\E \big (|\E_n (S_n)|^{p-2}\bar{S}_{n}^2 \big ) = \E \big (|\E_n (S_n)|^{p-2} \E_n ( \bar{S}_{n}^2 ) \big ) \leq a_n^{p-2} \Vert \E_0 ( S^2_n) \Vert_{p/2} \, .
$$
So, overall, we get  
$$
\Vert \E_n(S_n) + \bar{S}_{n} \Vert_p^p  \leq 2 a_{n}^{p}+ p \, a_{n}^{p-1} \Vert{\mathbb{E}}_{0}(S_{n}%
)\Vert_{p} + \frac{p(p-1)}{2} a_{n}^{p-2} \Vert{\mathbb{E}}_{0}(S_{n}^{2})\Vert_{p/2} \, ,
$$
proving \eqref{recurrence3} with $\delta=1$, $c=p$ and $c_2 = p(p-1)/2$. 

Assume now that $p\in]3,4[$. Inequality (86) in \cite{MP} (applied with $x=\E_n (S_n)$ and $y=\bar{S}_{n}$)  together with stationarity lead to
\[
\Vert \E_n(S_n) + \bar{S}_{n} \Vert_p^p \leq 2a_{n}^{p}+p a_{n}^{p-1}\Vert{\mathbb{E}}_{0}(S_{n})\Vert
_{p}+\frac{p(p-1)}{2} a_{n}^{p-2}\Vert{\mathbb{E}}_{0}(S_{n}^{2})\Vert
_{p/2}+2p(p-2)^{-1} \E \big (|\E_n (S_n)|  |\bar{S}_{n} |^{p-1} \big ) \,.
\]
To handle the last term in the right-hand side, we notice that for any $p \geq 3$ and any positive random variables $Y_0$ and $Y_1$ such that $\E ( Y_0^p) \leq a ^p$ and $\E ( Y_1^p) \leq a ^p$,
\beq \label{basicine}
\E ( Y_0 Y_1^{p-1}) \leq a^{p-2/(p-2)} \Vert \E ( Y_1 | Y_0)  \Vert_{p/2}^{1/(p-2)}
\eeq
(see the proof of inequality (83) in \cite{MP}). Using stationarity and applying \eqref{basicine} with $Y_0=|\E_n (S_n)|$ and $Y_1=  |\bar{S}_{n} |$, we get, for any $p \geq 3$,
\beq \label{consbasicine}
\E \big (|\E_n (S_n)|  |\bar{S}_{n} |^{p-1} \big ) \leq a_n^{p-2/(p-2)} \Vert \E_0 ( S_n)  \Vert_{p/2}^{1/(p-2)} \, .
\eeq
So, overall, for any $p \in ]3, 4[$,
\begin{multline} \label{consbasicine*}
\Vert \E_n(S_n) + \bar{S}_{n} \Vert_p^p \leq 2a_{n}^{p}+p a_{n}^{p-1}\Vert{\mathbb{E}}_{0}(S_{n})\Vert
_{p}+\frac{p(p-1)}{2} a_{n}^{p-2}\Vert{\mathbb{E}}_{0}(S_{n}^{2})\Vert
_{p/2} \\ + 2p(p-2)^{-1} a_n^{p-2/(p-2)} \Vert \E_0 ( S_n)  \Vert_{p/2}^{1/(p-2)} \,.
\end{multline}
But, for $p \geq 3$, $\Vert{\mathbb{E}}_{0}(S_{n}^{2})\Vert_{p/2} \leq a_n^{2-2/(p-2)} \Vert{\mathbb{E}}_{0}(S_{n}^{2})\Vert
_{p/2}^{1/(p-2)}$ which together with \eqref{consbasicine*} show that \eqref{recurrence3} holds with $\delta=1/(p-2)$, $c=p$ and $c_2 = p(p-1)/2 + 2p/(p-2)$. 

It remains to prove the inequality (\ref{recurrence3}) for $p\geq4$. Inequality \eqref{comparaison} (applied with $x=\E_n (S_n)$ and $y=\bar{S}_{n}$) together with stationarity lead to
\beq \label{1p>4}
\Vert \E_n(S_n) + \bar{S}_{n} \Vert_p^p \leq 2a_{n}^{p}+4^{p} \E \big (|\E_n (S_n)|^{p-1}  |\bar{S}_{n} | \big ) + 4^p \E \big (|\E_n (S_n)|  |\bar{S}_{n} |^{p-1} \big )  \,.
\eeq
Notice that H\"older's inequality combined with stationarity entails that 
\[
\E \big (|\E_n (S_n)|^{p-1}  |\bar{S}_{n} | \big ) = \E \big (|\E_n (S_n)|^{p-1}  \E_n (  |\bar{S}_{n} | ) \big )  \leq a_n^{p-1} \Vert \E_0 (|{S}_{n} | ) \Vert_p \,.
\]
But, by Jensen's inequality, $\Vert \E_0 (|{S}_{n} | ) \Vert_p \leq \Vert \E_0 ({S}^2_{n}  ) \Vert^{1/2}_{p/2}$. Hence, since $p \geq 4$, by using stationarity, we derive that 
\beq \label{2p>4}
\E \big (|\E_n (S_n)|^{p-1}  |\bar{S}_{n} | \big )  \leq a_n^{p-1} \Vert \E_0 ({S}^2_{n}  ) \Vert^{1/2}_{p/2} \leq  a_n^{p-2/(p-2)} \Vert \E_0 ({S}^2_{n}  ) \Vert^{1/(p-2)}_{p/2}\,.
\eeq
Therefore, starting from \eqref{1p>4} and using the bounds \eqref{consbasicine} and \eqref{2p>4}, we get 
\[
\Vert \E_n(S_n) + \bar{S}_{n} \Vert_p^p \leq 2a_{n}^{p}+ 2^{2p+1}  a_n^{p-2/(p-2)} \Vert \E_0 ({S}^2_{n}  ) \Vert^{1/(p-2)}_{p/2}\, ,
\]
proving \eqref{recurrence3} with $\delta=1/(p-2)$, $c=0$ and $c_2 = 2^{2p+1}$.
\end{proof}
\begin{proof}[Proof of Remark \ref{banachvalued}] As it is pointed out in the proof of Theorem \ref{Rosenthal1}, the remark will be proven with the help of Proposition \ref{maxinequality}, if we can show that 
\begin{equation*}
a_{2n}^{p}\leq 2a_{n}^{p}+c_{1}a_{n}^{p-1} \Vert S_n - {\mathbb{E}}%
_{n}(S_{n})\Vert_{p} +c_{2}a_{n}^{p-2\delta}\Vert{\mathbb{E}}_{0}( |S_{n}|_{\mathbb B}^{2})\Vert
_{p/2}^{\delta}\, , \label{recurrenceremark}%
\end{equation*}
where $a^p_n = \mathbb E (| S_n |_{\mathbb B}^p)$,  $c_{1}$ and $c_{2}$ are positive constants depending only on $p$ and $\delta =\min (1/2 , 1/(p-2))$. Indeed, the second term in the right-hand side of \eqref{inemaxna1} can be bounded by the last term in the right-hand side of \eqref{rosenthalremark}. To see this it suffices to use Jensen's inequality and the fact that $\delta\leq 1/2$. 

Starting from \eqref{recurrence2} (by replacing the absolute values by the norm $| \cdot |_{\mathbb B}$), we see that to prove the above recurrence formula it suffices to show that there exists a positive constant $c$  depending only on $p$ such that
$$
\Vert \E_n(S_n) + \bar{S}_{n} \Vert_p^p  \leq 2a_{n}^{p} +c a_{n}^{p-2\delta}\Vert{\mathbb{E}}_{0}(|S_{n}|_{\mathbb B}^{2})\Vert
_{p/2}^{\delta}\,.
$$
The difference at this step with the proof of Theorem \ref{Rosenthal1} is that the inequality \eqref{comparaison} is used whatever $p >2$ (in the case of real-valued random variables, we have used more precise inequalities  when $p\in ]2,4[$).
\end{proof}

\begin{proof}[Proof of Corollary \ref{corburk}] To prove the corollary, it suffices to show that 
for any $0<\delta\leq1$ and
any real $p>2$,
\begin{equation} \label{R1burk}
 2^r\left(  \sum_{k=0}^{r-1} \frac{
     \Vert{\mathbb{E}}_{0}(S_{2^k}^{2})\Vert_{p/2}^{\delta}}{2^{2\delta k/ p}}\right)^{p/(2\delta)}\ll 2^{r p/2}%
\Vert{\mathbb{E}}_{0}(X_{1}^{2})\Vert_{p/2}^{p/2}+2^{r p/2}\Big(\sum_{j=0}^{r-1}\frac{\Vert{\mathbb{E}}_{0}%
(S_{2^j})\Vert_{p}+\Vert S_{2^j}-{\mathbb E}_{2^j}(S_{2^j}) \Vert_p}{2^{j/2}}\Big)^{p} \, ,
\end{equation}
and to apply Theorem \ref{Rosenthal1}. 

To prove \eqref{R1burk}, we shall use similar arguments as those developed in the proof of Lemma 12 in \cite{MP}. Setting  $b_{n}=\Vert{\mathbb{E}}_{0}(S_{n}^{2})\Vert_{p/2}$, assume that we can prove  that, for any integer $n$,
\begin{equation} \label{R2burk}
b_{2n}\leq 2b_{n}+2b_{n}^{1/2}(\Vert{\mathbb{E}}_{0}(S_{n})\Vert_{p} + \Vert S_n - {\mathbb{E}}_{n}(S_{n})\Vert_{p})\,.
\end{equation}
Then, by recurrence on the first term, the above inequality will entail that for any positive integer $k$, 
\[
b_{2^{k}}\leq 2^{k}b_{1}+\sum_{j=0}^{k-1}2^{k-j}b_{2^{j}}^{1/2} \big ( \Vert
{\mathbb{E}}_{0}(S_{2^{j}})\Vert_{p} +\Vert S_{2^{j}}-
{\mathbb{E}}_{2^j}(S_{2^{j}})\Vert_{p}  \big ) \,.
\]
Next, with the notation $B_{k}=\max_{0\leq j\leq k}2^{-j}b_{2^{j}}$, it will follow that 
\[
B_{k}\leq2 \max \Big( b_{1},B_{k}^{1/2}\sum_{j=0}^{k-1}2^{-j/2}\big ( \Vert
{\mathbb{E}}_{0}(S_{2^{j}})\Vert_{p} +\Vert S_{2^{j}}-
{\mathbb{E}}_{2^j}(S_{2^{j}})\Vert_{p}  \big ) \Big )\,,
\]
implying that
\begin{equation*} \label{R2burkbis*}
2^{-k}b_{2^{k}}\leq B_{k}\leq 2b_{1 }+2^{2}\Big (\sum_{j=0}^{k-1}2^{-j/2}%
\big ( \Vert
{\mathbb{E}}_{0}(S_{2^{j}})\Vert_{p} +\Vert S_{2^{j}}-
{\mathbb{E}}_{2^j}(S_{2^{j}})\Vert_{p}  \big ) \Big)^{2} \, .
\end{equation*}
Since the above inequality clearly entails \eqref{R1burk}, to prove the corollary it then suffices to prove \eqref{R2burk}. With this aim, by using the notation $\bar{S}_{n}=X_{n+1}+\dots+X_{n}$, we first write that 
$S_{2n}^{2}=S_{n}^{2}+\bar{S}_{n}^{2}+2 \E_n (S_{n})\bar{S}_{n} + 2 ( S_n - \E_n (S_{n}))\bar{S}_{n}$. Hence, by
stationarity, 
\[
b_{2n}\leq 2 b_n + 2\Vert{\mathbb{E}}_{0}\big (\E_n (S_{n}) {\mathbb{E}}_{n}%
(\bar{S}_{n})\big)\Vert_{p/2} + 2\Vert{\mathbb{E}}_{0}\big ( (S_n - \E_n (S_{n})) \bar{S}_{n}\big)\Vert_{p/2}\,.
\]
Therefore the inequality \eqref{R2burk} follows from the following upper bounds: applying  Cauchy-Schwarz inequality twice and using stationarity, we get  
\begin{align*}
\Vert{\mathbb{E}}_{0}\big (\E_n( S_{n}) {\mathbb{E}}_{n}(\bar{S}_{n})\big)\Vert
_{p/2}
& \leq  \Vert {\mathbb{E}}_{0}
({\mathbb{E}}_{n}^{2}({S}_{n}))\Vert^{1/2}_{p/2} \times  \Vert {\mathbb{E}}_{0}%
({\mathbb{E}}_{n}^{2}(\bar{S}_{n}))\Vert^{1/2}_{p/2}  \\
& \leq  \Vert {\mathbb{E}}_{0}({S}^2_{n}))\Vert^{1/2}_{p/2} \times  \Vert
{\mathbb{E}}^2_{n}(\bar{S}_{n})\Vert^{1/2}_{p/2} \leq b_n^{1/2}\Vert{\mathbb{E}}_{0}(S_{n})\Vert_{p}\, ,
\end{align*}
and 
\begin{equation*}
\Vert{\mathbb{E}}_{0}\big ( (S_n - \E_n (S_{n}))
\bar{S}_{n}\big)\Vert_{p/2}\leq\Vert{\mathbb{E}}_{0} (( (S_{n} - \E_n (S_n) )^{2}) \Vert_{p/2 }^{1/2} \Vert {\mathbb{E}}_{0}( 
\bar{S}^2_{n})\Vert_{p/2}^{1/2} \\
\leq b_{n}^{1/2}\Vert S_n - {\mathbb{E}}_{n}(S_{n})\Vert_{p}\,.
\end{equation*}
\end{proof}

\begin{proof}[Proof of Remark \ref{remarkinemax}] Let $n$ and $r$ be integers such that $2^{r-1}\le n<  2^r$. Notice first that
\beq \label{inemaxna2p1} \Big \Vert \max_{1\le k\le n}|S_k|_{\mathbb B}\Big \Vert_p
\le\big \Vert \max_{1\le k\le 2^r}|S_m|_{\mathbb B}
  \Big \Vert_p \, \text{ and }\Vert S_{2^r}\Vert_p\le 2\Vert S_{2^{r-1}}\Vert_p\le 2\max_{1\le k\le n}\Vert S_k\Vert_p \,
  \eeq
(for the second inequality we use the stationarity). Now,
setting $V_m=\Vert{\mathbb E}_0 (S_m )  \Vert_p$, we have by stationarity
that for all $n,m \geq 0$, $V_{n+m} \leq V_n + V_m$ and then, according to the first item of Lemma 37 of \cite{MP},
\begin{align} \label{inemaxna2p2} 2^{r/p}\sum_{\ell=0}^{r-1}2^{-\ell /p}
    \Vert{\mathbb E}_0(S_{2^\ell})\Vert_p & \le
      n^{1/p}\frac{2^{1/p}2^{2+1/p}}{2^{1+1/p}-1}\sum_{k=1}^n
   \frac{\Vert {\mathbb E}_0(S_k)\Vert_p}{k^{1+1/p}} \nonumber
    \\ & \leq  n^{1/p}\frac{2^{1+1/p}}{1-2^{-1/p-1}}\sum_{k=1}^n
   \frac{\Vert {\mathbb E}_0(S_k)\Vert_p}{k^{1+1/p}} \, .
   \end{align}
On an other hand, let $W_m = \Vert S_{m} - {\mathbb E}_m (S_{m} ) \Vert_p $, and note that the following claim is valid:
 \begin{claim} \label{claimfacile}
If ${\mathcal F}$ and
${\mathcal G}$ are $\sigma$-algebras such that ${\mathcal G }\subset {\mathcal F}$, then
for any $X$ in ${\mathbb L}^p ({\mathbb B} )$ where $p\geq 1$, $\Vert X - {\mathbb E}(X | {\mathcal F}) \Vert_p \leq
2 \Vert X - {\mathbb E}(X | {\mathcal G}) \Vert_p$.
\end{claim}
The above claim together with the stationarity imply that for all
$n,m \geq 0$, $W_{n+m} \leq 2( W_n + W_m)$. Therefore, using once again the first item of
Lemma 37 of \cite{MP}, we get 
\begin{equation}  \label{inemaxna2p3}  2^{r/p}\sum_{\ell=0}^{r} 2^{-\ell/p} \Vert S_{2^{\ell}} - {\mathbb{E}}_{2^{\ell}}(S_{2^{\ell}}%
) \Vert_{p} \le 2 n^{1/p}\frac{2^{1+ 1/p}}
  { 1 - 2^{-1/p -1} }\sum_{\ell = 1}^{2n} \frac{\Vert S_{\ell} -
{\mathbb E}_{\ell} (S_{\ell} )  \Vert_p}{\ell^{1+1/p}} \, .
\end{equation}
The inequality \eqref{inemaxna2} then follows from  the inequality (\ref{inemaxna1}) by taking into account the upper bounds \eqref{inemaxna2p1}, \eqref{inemaxna2p2} and \eqref{inemaxna2p3}.
\end{proof}

\subsection{A tightness criterion}

We begin with the definition of the number of brackets of a family of functions.

\begin{defi}\label{defnb}
Let $P$ be a probability measure on a measurable space ${\mathcal X}$.
For any measurable function $f$ from ${\mathcal X}$ to ${\mathbb
R}$, let $\|f\|_{P, 1}=P(|f|)$. If $\|f\|_{P, 1}$ is finite,
one says that $f$ belongs to $L_P^1$. Let ${\mathcal F}$ be some subset
of $L_P^1$. The number of brackets ${\mathcal N}_{P,1}(\varepsilon,
{\mathcal F})$ is the smallest integer $N$ for which there exist some
functions $f_1^-\leq f_1, \ldots , f_N^-\leq f_N$ in ${\mathcal F}$ such
that: for any integer $1\leq i \leq N$ we have
$\|f_{i}-f_i^-\|_{P,1} \leq \varepsilon$, and for any function $f$
in ${\mathcal F}$ there exists an integer $1\leq i \leq N$ such that
$f_i^- \leq f \leq f_{i}$.
\end{defi}

Proposition \ref{p2} below gives a general tightness criterion for empirical processes. Its proof is based on a decomposition given in  \cite{AP} (see also \cite{DePr}). Under the setting and conditions of  Theorem \ref{Kiefer}, the criterion \eqref{squinousfaut} will be shown to hold with the help of Proposition \ref{consdirect} (see the proof of Theorem \ref{Kiefer} in Section \ref{lastsec}).
\begin{prop}\label{p2}
Let $(X_i)_{i \geq 1}$ be a  sequence of identically distributed random variables
with values in a measurable space ${\mathcal X}$,  with common distribution $P$. Let $P_n$
be the empirical measure
$P_n=n^{-1} \sum_{i=1}^n \delta_{X_i}$, and let $S_n$ be the  empirical
process $S_n=n (P_n-P)$.
 Let
${\mathcal F}$ be a class of functions from ${\mathcal X}$ to ${\mathbb R}$ and
${\mathcal G}=\{f-l, (f,l) \in {\mathcal F}\times {\mathcal F} \}$.
Assume that there exist  $r \geq 2$,   $p >2$  and $C >0$ such that for any function $g$ of ${\mathcal G}\cup {\mathcal F}$ and any positive integer $n$,  we have
\begin{equation}\label{squinousfaut}
  \Big  \|\max_{1 \leq k \leq n} |S_k(g)| \Big \|_p \leq C( \sqrt n \|g\|^{1/r}_{P,1} + n^{1/p}) \, ,
\end{equation}
where $S_k(g):=\sum_{i=1}^k(g(X_i)-P(g))$. If moreover
$$
\int_0^1 x^{(1-r)/r} ({\mathcal N}_{P,1}(x, {\mathcal F}))^{1/p} dx < \infty
\quad and  \quad \lim_{x \rightarrow 0} x^{p-2}  {\mathcal N}_{P,1}(x, {\mathcal F}) =0 \, ,
$$
then
\begin{align}\label{equiibis}
&\lim_{\delta \rightarrow 0}
   \limsup_{n \rightarrow \infty} {\mathbb E}\Bigl( \max_{1\leq k \leq n}
    \sup_{g \in {\mathcal G}, \|g\|_{P,1} \leq \delta} n^{-p/2}|S_k(g)|^p \Bigr)=0 \,  ,\\
    \label{equiiter}
and \quad
& \lim_{\delta \rightarrow 0} \limsup_{n \rightarrow \infty} \frac 1 \delta
{\mathbb E}\Bigl( \max_{1\leq k \leq [n\delta]}
    \sup_{f \in {\mathcal F}} n^{-p/2}|S_k(f)|^p \Bigr)=0 \, .
\end{align}
\end{prop}

\begin{proof}[Proof of Proposition \ref{p2}]
It is almost the same as that of Proposition 6 in \cite{DePr}. Let us
only give the main steps.

For any positive integer $k$, denote by ${\mathcal N}_k={\mathcal
N}_{P,1}(2^{-k},{\mathcal F})$ and by ${\mathcal F}_k$ a family of functions
$f_1^{k,-}\leq f_1^k,   \ldots , f^{k,-}_{{\mathcal N}_k}\leq
f^{k}_{{\mathcal N}_k}$ in ${\mathcal F}$ such that
$\|f_i^{k}-f_{i}^{k,-}\|_{P,1} \leq 2^{-k}$, and for any $f$ in
${\mathcal F}$, there exists an integer $1\leq i \leq {\mathcal N}_k$ such
that $f^{k,-}_{i}\leq f \leq f^k_{i}$.

We follow exactly the proof of Proposition 6 in \cite{DePr}. For reader's convenience, we give the key details. For any  $f$ in ${\cal F}$, there exist two functions $g_k^-$ and $g_k^+$ in ${\cal F}_k$ such that $g_k^-\leq f \leq g_k^+$ and $\|g_k^+-g_k^-\|_{P,1} \leq 2^{-k}$. Hence, for any $1 \leq j \leq n$, 
$$
  S_j(f)-S_j(g_k^-)  \leq   S_j(g_k^+)-S_j(g_k^-) +\sum_{i=1}^j {\mathbb E}((g_k^+-f)(X_i)) \leq  |S_j(g_k^+)-S_j(g_k^-)| + j 2^{-k} \, .
$$
Since $g_k^- \leq f$, we also have that $S_j(g_k^-)-S_j(f) \leq j 2^{-k}$, which enables us to conclude that
$$
|S_j(f)-S_j(g_k^-)| \leq   |S_j(g_k^+)-S_j(g_k^-)| + j 2^{-k}\, .
$$
Consequently
\begin{equation}\label{5.3}
  \sup_{f \in {\cal F}}|S_j(f)-S_j(g_k^-)| \leq \max_{1\leq i \leq {\cal N}_k}|S_j(f^k_i)-S_j(f^{k,-}_{i})|+j 2^{-k} \, .
\end{equation}
Notice now the following elementary fact: given $N$ real-valued random variables $Z_1, \dots, Z_N$, we have
\begin{equation}\label{5.1}
    \|\max_{1\leq i \leq N} |Z_i|\|_p \leq N^{1/p} \max_{1\leq i \leq N} \|Z_i\|_p \, .
\end{equation}
Combining (\ref{5.1}) and (\ref{5.3}), we obtain 
\begin{equation}\label{5.4}
\Bigl\|\max_{1 \leq j \leq n}\sup_{f \in {\cal F}}|S_j(f)-S_j(g_k^-)|\Bigr\|_p \leq {\cal
N}_k^{1/p} \max_{1\leq i \leq {\cal
N}_k}\|\max_{1 \leq j \leq n} | S_j(f_i^k)-S_j(f_{i}^{k,-}) | \|_p + n  2^{-k}\, .
\end{equation}
Starting from (\ref{5.4}) and applying \eqref{squinousfaut}, we obtain
\begin{equation}\label{5.5}
\Bigl\|\max_{1 \leq j \leq n}\sup_{f \in {\cal F}}n^{-1/2}|S_j(f)-S_j(g_k^-)|\Bigr\|_p \leq C({\cal N}_k^{1/p}2^{-k/r}+{\cal N}_k^{1/p}n^{1/p-1/2})
+\sqrt{n} 2^{-k}\, .
\end{equation}
By the arguments developed right after the inequality (4.6) in \cite{DePr}, we infer that there exists a sequence $h_{k(n)}(f)$ belonging to  ${\cal F}_{k(n)}$ such that
\begin{equation}\label{5.2}
    \lim_{n \rightarrow \infty} \Bigl\|\max_{1 \leq j \leq n}\sup_{f \in {\cal F}}n^{-1/2}|S_j(f)-S_j(h_{k(n)}(f))|\Bigr\|_p=0 \, .
\end{equation}

We prove now that for any $\varepsilon>0$, there exist
 $ N(\varepsilon)$ and $m=m(\varepsilon)$ such that : for any $n\geq N(\varepsilon)$
there exists a function
$f_{n,m}$ in ${\cal F}_m$ such that
\begin{equation}\label{5.7}
 \Bigl\|\max_{1 \leq j \leq n}\sup_{f \in {\cal F}}n^{-1/2}|S_j(f_{n,m})-S_j(h_{k(n)}(f))|\Bigr\|_p \leq  \varepsilon \, .
\end{equation}
Given  $h$ in ${\cal F}_k$, choose a function $T_{k-1}(h)$ in ${\cal F}_{k-1}$ such that
 $\|h-T_{k-1}(h)\|_{P,1} \leq 2^{-k+1}$. Denote by $\pi_{k,k}=Id$ and for $l<k$,  $\pi_{l,k}(h)=T_l \circ \cdots \circ T_{k-1}(h)$.
We consider the function $f_{n,m}=\pi_{m, k(n)}(h_{k(n)}(f))$. For the sake of brevity, we write $h_{k(n)}$ instead of  $h_{k(n)}(f)$.
We  have that
\begin{equation}\label{5.8}
\Bigl \|\max_{1 \leq j \leq n} \sup_{f \in {\cal F}} |S_j(f_{n,m})-S_j(h_{k(n)})|\Bigr\|_p \leq
\sum_{l=m+1}^{k(n)}\Bigl \|\max_{1 \leq j \leq n} \sup_{f \in {\cal F}} |S_j(\pi_{l, k(n)}(h_{k(n)}))-S_j(\pi_{l-1, k(n)}(h_{k(n)}))|\Bigr\|_p\, .
\end{equation}
Clearly
$$
  \Bigl \|\max_{1 \leq j \leq n} \sup_{f \in {\cal F}} |S_j(\pi_{l, k(n)}(h_{k(n)}))-S_j(\pi_{l-1, k(n)}(h_{k(n)}))|\Bigr\|_p \leq
\Bigl \|\max_{1 \leq j \leq n} \max_{f \in {\cal F}_l} |S_j(f)-S_j(T_{l-1}(f))|\Bigr\|_p\, .
$$
Using then \eqref{squinousfaut} combined with \eqref{5.1}, it follows that 
$$
\Bigl \|\max_{1 \leq j \leq n} \sup_{f \in {\cal F}} n^{-1/2} |S_j(f_{n,m})-S_j(h_{k(n)})|\Bigr\|_p  \leq  C
\sum_{l=m+1}^{k(n)}(2^{1/r}{\cal N}_l^{1/p}2^{-l/r}+ {\cal N}_l^{1/p}n^{1/p-1/2})
 \, .
$$
To complete the proof of \eqref{5.7} we use the same arguments as in \cite{DePr}, 
page 130.

Combining \eqref{5.2} and \eqref{5.7}, it follows that for any $\varepsilon>0$, there exist
 $ N(\varepsilon)$ and $m=m(\varepsilon)$ such that: for any $n\geq N(\varepsilon)$
there exists $f_{n,m}$ in ${\mathcal F}_m$  for which
\begin{equation}\label{mainpoint}
   \Bigl\| \max_{1 \leq k \leq n} \sup_{f \in  {\mathcal F}}n^{-1/2}|S_k(f)-S_k(f_{n,m})|\Bigr\|_p \leq 2\varepsilon \, .
\end{equation}
Using the same argument as in \cite{AP} (see the paragraph ``Comparison of pairs" page 124),
we obtain 
$$
\Bigl \| \max_{1 \leq k \leq n}\sup_{f , g \in {\mathcal F} \atop \|f-g\|_{P,1} \leq \delta} n^{-1/2}|S_k(f)-S_k(g)|\Bigr \|_p \leq 8 \varepsilon + {\cal N}_m^{2/p}
\sup_{f , g \in {\mathcal F}  \atop \|f-g\|_{P,1} \leq 2 \delta}\Bigl\|  \max_{1 \leq k \leq n} n^{-1/2}|S_k(f)-S_k(g)|\Bigr\|_p  \, .
$$
Since by \eqref{squinousfaut}, 
$$
\sup_{f , g \in {\mathcal F}  \atop \|f-g\|_{P,1} \leq 2 \delta}\Bigl\|  \max_{1 \leq k \leq n} n^{-1/2}|S_k(f)-S_k(g)|\Bigr\|_p  \leq C( (2 \delta)^{1/r} + n^{1/p-1/2}) \, ,
$$
it follows that
$$
\Bigl \| \max_{1 \leq k \leq n}\sup_{f , g \in {\mathcal F} \atop \|f-g\|_{P,1} \leq \delta} n^{-1/2}|S_k(f)-S_k(g)|\Bigr  \|_p \leq 8 \varepsilon + C {\cal N}_m^{2/p} ( (2 \delta)^{1/r} + n^{1/p-1/2})  \, ,
$$
which proves (\ref{equiibis}).

Let us now  prove (\ref{equiiter}). We apply (\ref{mainpoint}) with
$\varepsilon=1$:  for $n\geq \delta^{-1} N(1)$, we infer from
(\ref{mainpoint}) that there exists $f_{[n\delta],m}$ in ${\mathcal F}_m$  for which
$$
\Bigl\| \max_{1 \leq k \leq [n\delta]} \sup_{f \in  {\mathcal F}}n^{-1/2}|S_k(f)-S_k(f_{[n\delta],m})|\Bigr\|_p \leq \sqrt{\delta} \, .
$$
Hence
\begin{equation}\label{un*}
\Bigl\| \max_{1 \leq k \leq [n\delta]} \sup_{f \in  {\mathcal F}}n^{-1/2}|S_k(f)|\Bigr\|_p
\leq \sqrt{\delta} + \Bigl\| \max_{1 \leq k \leq [n\delta]}
\sup_{f \in  {\mathcal F}}n^{-1/2}|S_k(f_{[n\delta],m})|\Bigr\|_p \, .
\end{equation}
Now, since ${\mathcal F}_m$ contains $2{\mathcal N}_m$ functions $(g_\ell)_{\ell \in \{1,\dots,2{\mathcal N}_m\}}$
(each $g_\ell$ being  one of the functions $f_i^m$ or $f_i^{m,-}$ in ${\mathcal F}_m$), it follows that
$$
\Bigl\| \max_{1 \leq k \leq [n\delta]}
\sup_{f \in  {\mathcal F}}n^{-1/2}|S_k(f_{[n\delta],m})|\Bigr\|_p
\leq \sum_{\ell=1}^{2 {\mathcal N}_m}\frac 1 {\sqrt n}
\Big\| \max_{1 \leq k \leq [n\delta]} |S_k(g_{\ell})| \Big \|_p \, .
$$
Let $K_m=\max_{f \in {\mathcal F}_m} \|f\|_{P,1}$. Applying (\ref{squinousfaut}),
we infer that
\begin{equation}\label{deux*}
\Bigl\| \max_{1 \leq k \leq [n\delta]}
\sup_{f \in  {\mathcal F}}n^{-1/2}|S_k(f_{[n\delta],m})|\Bigr\|_p
\leq
2C {\mathcal N}_m (K_{m}^{1/r}\sqrt{\delta}+ n^{-(p-2)/2p}\delta^{1/p})\, .
\end{equation}
Since $m=m(1)$ is fixed, (\ref{equiiter}) follows from (\ref{un*}) and (\ref{deux*}) and the fact that $p>2$. 
\end{proof}

\section{Inequalities for ergodic torus automorphisms}\label{ineqauto}
\setcounter{equation}{0}

In this section, we keep the same notations as in the introduction.
Let us denote by $E_u$, $E_e$ and $E_s$ the $S$-stable vector spaces associated to the eigenvalues
of $S$ of modulus respectively larger than one, equal to one and smaller than one.
Let $d_u$, $d_e$ and $d_s$ be their respective dimensions.
Let $v_1,...,v_d$ be a basis of $\mathbb R^d$ such that $v_1,...,v_{d_u}$ are in
$E_u$, $v_{d_u+1},...,v_{d_u+d_e}$ are in $E_e$ and $v_{d_u+d_e+1},...,v_d$ are in $E_s$.
We suppose moreover that $\textrm{det}(v_1|v_2|\cdots|v_d)=1$.
Let  $\Vert \cdot \Vert $ be the norm on $\mathbb R^d$ given by
$$\Bigl \Vert\sum_{i=1}^dx_iv_i\Bigr \Vert=\max_{i=1,...,d}|x_i| $$
and  $d_0(\cdot,\cdot)$ be the metric induced by $\Vert \cdot \Vert$ on $\mathbb R^d$.
Let also $d_1$ be the metric induced by $d_0$ on $\mathbb T^d$ namely, 
$$
d_1 ( \bar x,\bar y ) = \inf_{z \in {\mathbb Z}^d} d_0(x+z,y) \, .
$$
We define now
$B_u(\delta):=\{y\in E_u\ :\ \Vert y  \Vert \le\delta\}$,
$B_e(\delta):=\{y\in E_e\ :\  \Vert y \Vert \le\delta\}$
and $B_s(\delta)=\{y\in E_s\ :\  \Vert y \Vert \le\delta\}$.
For every $f:\mathbb T^d\rightarrow\mathbb R$, we consider the
moduli of continuity defined by: for every $\delta>0$, 
\begin{equation}
\omega(f,\delta):=\sup_{\bar x,\bar y\in\mathbb T^d\, :\, d_1(\bar x,\bar y)
   \le\delta}|f(\bar x)-f(\bar y)| \, ,
\end{equation}
$$\omega_{(s,e)}(f,\delta)=\sup\{|f(\bar x)-f(\bar x+\overline{h_s}+\overline{h_e})|,\ \bar
 x\in\mathbb T^d,\ h_s\in B_s(\delta),
    \ h_e\in B_e(\delta)\}$$
and
$$ \omega_{(u)}(f,\delta)=\sup\{|f(\bar x)-f(\bar x+\overline{h_u})|,\ \bar
 x\in\mathbb T^d,\ h_u\in B_u(\delta)\} \, .$$
Let $r_u$ be the spectral radius of $S^{-1}_{|E_u}$. For every
$\rho_u\in(r_u,1)$, there exists $K>0$ such that, for every integer $n\ge 0$, we have
\begin{equation}\label{hyp-a}
\forall h_u\in E_u,\ \
     \Vert S^{-n} h_u \Vert \le K\rho_u^{n} \Vert h_u \Vert 
\end{equation}
and
\begin{equation}\label{hyp-b}
\forall (h_e,h_s)\in E_e\times E_s,\ \
     \Vert S^n (h_e+h_s) \Vert \le K n^{d_e} \Vert h_e+h_s \Vert  \, .
\end{equation}

The following inequality can be viewed as an extension to continuous functions of a result for
H\"older functions established in \cite{SLBFP} but with a $\sigma$-algebra satisfying $\mathcal F_0\subseteq T^{-1}\mathcal F_0$ (this condition is not satisfied 
in the construction of $\mathcal F_0$ considered in \cite{SLBFP}). For the next result, we shall then use the construction of $\mathcal F_0$ given 
in \cite{Lind,SLB} combined with some arguments developed in \cite{SLBFP}.
\begin{thm}\label{THM2}
Let $\rho_u\in(r_u,1)$ and
$\zeta\in (\rho_u^{ 1/(3(d+2)(d_e+d_s))},1)$.
There exist $C>0$, $N\ge 0$, $\xi\in(0,1)$, a sequence of measurable sets $(\mathcal V_n)_{n\ge 0}$
and a $\sigma$-algebra $\mathcal F_0$ such that
$\mathcal F_0\subseteq T^{-1}\mathcal F_0$ and such that, for every bounded
$\varphi:\mathbb T^d\rightarrow\mathbb R$ and every integer $n\ge N$, we have
\begin{equation}\label{THM2a}
\left\Vert{\mathbb E}[\varphi|\mathcal F_n]-\varphi\right\Vert_\infty\le \omega_{(u)}
     (\varphi,\rho_u^n) \, ,
\end{equation}
\begin{equation}\label{THM2b}
\mbox{on }\mathcal V_n,\ \
   \left|{\mathbb E}[\varphi|\mathcal F_{-n}]-{\mathbb E}[\varphi]\right|
        \le C(\Vert\varphi\Vert_\infty\xi^n+\omega_{(s,e)}(\varphi,\zeta^n))
\end{equation}
and
\begin{equation}\label{THM2c}
\bar\lambda(\mathbb T^d\setminus\mathcal V_n)\le C\xi^n \, ,
\end{equation}
where $\mathcal F_{k}:=T^{-k}\mathcal F_0$ for every $k\in\mathbb Z$.
\end{thm}
\begin{rqe} \label{remarqueTH2}
With the notations of Theorem \ref{THM2}, (\ref{THM2b}) and (\ref{THM2c}) imply that, for
every $p\ge 1$ and every $(\rho_u,\zeta)$ as in Theorem \ref{THM2},
there exists $c_p$ such that, for every bounded $\varphi:{\mathbb T}^d\rightarrow\mathbb R$
and every integer $n\ge 0$, we have
\begin{equation}\label{THM2d}
\forall n\ge 0,\ \ \
   \left\Vert{\mathbb E}[\varphi|\mathcal F_{-n}]-{\mathbb E}[\varphi]\right\Vert_{p}
        \le c_p(\Vert\varphi\Vert_\infty\xi^{\frac n p}+\omega_{(s,e)}(\varphi,\zeta^n)) \, .
\end{equation}
\end{rqe}
The remainder of this section is devoted to the proof of
Theorem \ref{THM2} and to the statements and the proofs of some preliminary results.   
Let $\rho_u\in(r_u,1)$ and $K$ satisfying (\ref{hyp-a}) and (\ref{hyp-b}).
Let $m_u$, $m_e$, $m_s$ be the Lebesgue measure on $E_u$ (in the basis $v_1,...,v_{d_u}$),
$E_e$ (in the basis $v_{d_u+1},...,v_{d_u+d_e}$) and $E_s$ (in the basis
$v_{d_u+d_e+1},...,v_d$) respectively.
We observe that
$d\lambda(h_u+h_e+h_s)=dm_u(h_u)dm_e(h_e)dm_s(h_s).$

The properties satisfied by the filtration considered in \cite{Lind,SLB}
and enabling the use of Gordin's method
will be crucial here. Given a finite partition $\mathcal P$ of $\mathbb T^d$, we define
the measurable partition $\mathcal P_0^{\infty}$ by:
$$\forall \bar x\in\mathbb T^d,\ \
 \mathcal P_0^{\infty}(\bar x):=\bigcap_{k\ge 0}T^k\mathcal P(T^{-k}(\bar x)) \, .$$
Next, for every integer $n$, we consider the $\sigma$-algebras $\mathcal F_n$ generated by
$$\forall \bar x\in\mathbb T^d,\ \
  \mathcal P_{-n}^{\infty}(\bar x):=\bigcap_{k\ge -n}T^k\mathcal P(T^{-k}(\bar x))
      =T^{-n}(\mathcal P_0^{\infty}(T^n(\bar x)) \, .$$
We obviously have $\mathcal F_n=T^{-n}\mathcal F_0\subseteq\mathcal F_{n+1}=T^{-1}\mathcal F_n$.
Let $r_0>0$ be such that $(h_u,h_e,h_s)\mapsto \overline{h_u+h_e+h_s}$ defines a diffeomorphism
from $B_u(r_0)\times B_e(r_0)\times B_s(r_0)$ on its image in $\mathbb T^d$.
Observe that, for every $\bar x\in\mathbb T^d$,
on the set $\bar x+B_u(r_0)+B_e(r_0)+B_s(r_0)$, we have
$d\bar\lambda(\bar x+\overline{h_u}+\overline{h_e}+\overline{h_s})=dm_u(h_u)dm_e(h_e)dm_s(h_s).$
\begin{prop}[\cite{Lind,SLB} applied to $T^{-1}$, see also \cite{DeMePene1}]\label{partition}
There exist some $Q>0$ and some finite partition $\mathcal P$ of $\mathbb T^d$
whose elements are of the form
$\sum_{i=1}^dI_i\overline{v_i}$ where the $I_i$ are intervals with diameter smaller than
$\min(r_0,K)$ such that,
for almost every $\bar x\in\mathbb T^d$,
\begin{itemize}
\item the local leaf $\mathcal P_0^{\infty}(\bar x)$
of $\mathcal P_0^{\infty}$ containing $\bar x$ is a set $\bar x+\overline{F_{\bar x}}$, with $0\in F_{\bar x}\subseteq E_u$ and such that
$F_{\bar x}$ is a uniformly bounded convex set having non-empty interior in $E_u$,
\item we have, for all $n\in\mathbb Z$,
$${\mathbb E}[f|\mathcal F_n](\bar x)=\frac 1{m_u(S^{-n}F_{T^n\bar x})}
                 \int_{S^{-n}F_{T^n\bar x}}f(\bar x+\overline{h_u})\, dm_u(h_u) \, ,$$
\item for every $\gamma>0$, we have
$$m_u(\partial (F_{\bar x})(\gamma))\le Q\gamma \, ,  $$
where
$$\partial \mathcal C(\beta):=\{y\in \mathcal C \ :\ d_0(y,\partial \mathcal C)\le\beta \}  \ \text{ for any $\mathcal C \subseteq E_u$}\, . $$
\end{itemize}
\end{prop}
Recall now an exponential decorrelation result for Lipschitz continuous functions.
\begin{prop}[\cite{Lind} and also section 4.1 of \cite{FPESAIM}]\label{decorr}
There exist $C_0>0$ and $\xi_0 \in (0,1)$ such that, for every nonnegative integer $n$ and every Lipschitz continuous functions
$f,g:\mathbb T^d\rightarrow\mathbb C$ with $\int_{\mathbb T^d}g\, d\bar\lambda=0$, we have
$$\left|\int_{\mathbb T^d}(f.g\circ T^n)\, d\bar\lambda\right|
   \le C_0 ( \Vert f \Vert_\infty \Vert g \Vert_\infty+ \Vert f \Vert_\infty Lip(g)+
      \Vert g \Vert_\infty Lip(f))\xi_0^n \, ,$$
where $Lip(h)$ is the Lipschitz constant of $h$.
\end{prop}
Let $Q$ be the constant appearing in Proposition \ref{partition}.
The following result is an adaptation of Proposition 1.3 of \cite{SLBFP}.
\begin{prop}\label{PROP1}
Let $\zeta_1\in (\xi_0^{1/((d+2)(d_e+d_s))},1)$ where $\xi_0$ is given in Proposition \ref{decorr}.
There exist $C_1>0$, $N_1\ge 1$ and $\xi_1\in(0,1)$ such that,
for every $\bar\lambda$-centered bounded function
$\varphi:\mathbb T^d\rightarrow\mathbb R$, every $\bar x\in\mathbb T^d$, every $n\ge N_1$
and every bounded convex set
$\mathcal C\subseteq E_u$ with diameter smaller than $r_0$, satisfying
$m_u(\partial \mathcal C(\beta))\le Q\beta$ (for every $\beta>0$), we have
$$\left|\frac 1{m_u(S^n\mathcal C)}\int_{S^n\mathcal C}\varphi(\bar x+\overline{h_u})\, dm_u(h_u)\right|
 \le K_1\left(\frac{ \Vert \varphi \Vert_\infty\xi_1^n}{m_u(\mathcal C)}+\omega(\varphi,\zeta_1^n)\right) \, .$$
\end{prop}
\begin{proof}
Let $r:=\xi_0^{-1/(d+2)}$.
We take $\varepsilon_n=\alpha^n$ with $\alpha\in(0,1)$ such that
$\zeta_1>\alpha>\xi_0^{1/((d+2)(d_e+d_s))}\ge r^{-1}$ and $n$ such that $\alpha^n <r_0$.
Let $U:=T^{-n}\bar x+\overline{{\mathcal C}+B_s(\varepsilon_n)+B_e(\varepsilon_n)}$. We have
$T^n(U)=\bar x+\overline{S^n{\mathcal C}+S^nB_s(\varepsilon_n)+S^nB_e(\varepsilon_n)}$.
We have
\begin{eqnarray*}
\int_{\mathbb T^d}{\mathbf 1}_{T^nU}.\varphi\, d\bar\lambda
&=&\int_{{\mathcal C}\times B_e(\varepsilon_n)\times B_s(\varepsilon_n)}
     \varphi(T^n(T^{-n}\bar x+\overline{h_u}+\overline{h_e}+\overline{h_s}))\,
    dm_u(h_u)dm_e(h_e)dm_s(h_s)\\
&=&\int_{\mathcal V_n}\varphi(\bar x+\overline{h_u}+\overline{h_e}+\overline{h_s})\, dm_u(h_u)dm_e(h_e)dm_s(h_s) \, ,
\end{eqnarray*}
with $\mathcal V_n:=S^n{\mathcal C}\times S^nB_e(\varepsilon_n)\times S^nB_s(\varepsilon_n)$.
Moreover we have
$$\int_{S^n {\mathcal C}}\varphi(\bar x+\overline{h_u})\, dm_u(h_u) =  \frac{1}{m_s(S^n(B_s(\varepsilon_n))
        m_e(S^n(B_e(\varepsilon_n))}\int_{\mathcal V_n}
    \varphi(\bar x+\overline{h_u})\, dm_u(h_u)dm_e(h_e)dm_s(h_s) \, .$$
Hence, due to (\ref{hyp-b}), we have
$$\left| \int_{\mathbb T^d}
  {\mathbf 1}_{T^nU}.\varphi\, d\bar\lambda-m_s(S^n(B_s(\varepsilon_n))
        m_e(S^n(B_e(\varepsilon_n))\int_{S^n {\mathcal C}}\varphi(\bar x+\overline{h_u})\, dm_u(h_u)\right|
    \le \bar\lambda(U)\omega_{(s,e)}(\varphi,Kn^{d_e}\varepsilon_n) \, .$$
Since $\bar\lambda(U)=m_u(S^n{\mathcal C})m_s(S^n(B_s(\varepsilon_n))
        m_e(S^n(B_e(\varepsilon_n))$, we get, for $n$ large enough
(that is, such that $Kn^{d_e}\varepsilon_n\le\zeta_1^n$),
\begin{eqnarray*}
\left| \frac 1{\bar\lambda(U)}\int_{\mathbb T^d}{\mathbf 1}_{T^nU}\varphi\, d\bar\lambda-
     \frac 1{m_u(S^n{\mathcal C})}\int_{S^n {\mathcal C}}\varphi(\bar x+\overline{h_u})\, dm_u(h_u)\right|
    &\le& \omega_{(s,e)}(\varphi,Kn^{d_e}\varepsilon_n)\\
    &\le& \omega_{(s,e)}(\varphi,\zeta_1^n) \, .
\end{eqnarray*}
For every $n\ge 0$ and $\bar x\in\mathbb T^d$, we define
$\chi_n(\bar x):=(d+1)2^{-d}r^{n(d+1)} d_1(\bar x,\mathbb T^d\setminus B(0,r^{-n}))$, where $B(0,r^{-n}) = \{ \bar x \in \mathbb T^d \, , \, d_1 ( \bar 0 , \bar x) \leq r^{-n} \}$.      
Let us observe that $\chi_n$
is a nonnegative $(d+1)r^{n(d+1)}2^{-d}$-Lipschitz continuous function
supported in $B(0,r^{-n})$, uniformly bounded by $(d+1)2^{-d} r^{nd}$
and such that $\int_{\mathbb T^d}\chi_n\, d\bar\lambda=1$.
We will denote by $*$ the usual convolution product with respect to $\bar\lambda$.
We will estimate
$$ \left|\int_{\mathbb T^d} {\mathbf 1}_U\circ T^{-n}.\varphi\, d\bar\lambda
     -\int_{\mathbb T^d} (\chi_n*{\mathbf 1}_U)\circ T^{-n}.(\chi_n*\varphi)) \, d\bar\lambda\right| \, . $$
First observe that
\begin{equation}\label{ZZZ1}
\left|\int_{{\mathbb T}^d}(\chi_n*\mathbf 1_U)\circ T^{-n}.(\chi_n*\varphi-\varphi)\, d\bar\lambda\right|
    \le \omega(\varphi,r^{-n})\bar\lambda(U) \, .
\end{equation}
Second, we have
\begin{equation}\label{ZZZ2}
\left|\int_{{\mathbb T}^d}(\chi_n*\mathbf 1_U-\mathbf 1_U)\circ T^{-n}.\varphi\, d\bar\lambda\right|
    \le \Vert\varphi\Vert_\infty\int_{{\mathbb T}^d}|\chi_n*\mathbf 1_U-\mathbf 1_U|d\bar\lambda \, ,
\end{equation}
and let us prove that
\begin{equation}\label{ZZZ3}
\int_{{\mathbb T}^d}|\chi_n*\mathbf 1_U-\mathbf 1_U|d\bar\lambda\le 3\bar\lambda(\partial U(r^{-n})) \, .
\end{equation}
To see this, observe that $\chi_n(\bar t)\mathbf 1_U(\bar x-\bar t)-\mathbf 1_U(\bar x)=(\chi_n(\bar t)-1)\mathbf 1_U(\bar x)$ 
except if $\mathbf 1_U(\bar x-\bar t)\ne\mathbf 1_U(\bar x)$ and if $\bar t\in B(0, r^{-n})$.
Hence $\chi_n*\mathbf 1_U(\bar x)\ne\mathbf 1_U(\bar x)$ implies either that $\bar x\in \partial U(r^{-n})$ where $\partial U(r^{-n} ):= \{x \in U \, : \, d_1(x, \partial U) < r^{-n} \}$, or
that $\bar x$ belongs to the set $U'$ of points such that
$\bar x\not\in U$ but there exists $\bar t_0\in B(0,r^{-n})$ such that $\bar x-\bar t_0\in U$.

On the one hand, we have
\begin{eqnarray}
\int_{\partial U(r^{-n})}|\chi_n*\mathbf 1_U-\mathbf 1_U|\, d\bar\lambda
   &\le&\int_{\partial U(r^{-n})}\left(\int_{\mathbb T^d}\chi_n(\bar t)\mathbf 1_U(\bar x-\bar t)\, d\bar\lambda( \bar t)\right)\,
         d\bar\lambda(\bar x) +\bar\lambda(\partial U(r^{-n}))\nonumber\\
   &\le&\bar\lambda(\partial U(r^{-n}))
 \int_{\mathbb T^d}\chi_n(\bar t) d\bar\lambda( \bar t)+\bar\lambda(\partial U(r^{-n}))\nonumber\\
   &\le& 2\bar\lambda(\partial U(r^{-n}))\label{ZZZ3a},
\end{eqnarray}
using the fact that $\chi_n$ is  nonnegative with unit integral.
On the other hand, we have
\begin{eqnarray}
\int_{U'}|\chi_n*\mathbf 1_U-\mathbf 1_U|\, d\bar\lambda
   &\le& \int_{U'}\left(\int_{\mathbb T^d}\chi_n(\bar t)\mathbf 1_U(\bar x-\bar t)\, d\bar\lambda(\bar t)\right)\,
         d\bar\lambda(\bar x)\nonumber\\
    &\le& \int_{\mathbb T^d\setminus U}\left(\int_{\bar t:\bar x-\bar t\in U}\chi_n(\bar t)\, d\bar\lambda(\bar t)\right)\,
         d\bar\lambda( \bar x)\nonumber\\
    &\le& \int_{\mathbb T^d}\left(\int_{\partial U(r^{-n})}\chi_n(\bar x-\bar s)\, d\bar\lambda(\bar s)\right)\,
         d\bar\lambda(\bar x)\nonumber\\
    &\le& \int_{\partial U(r^{-n})}\left(\int_{\mathbb T^d}\chi_n(\bar x-\bar s)\, d\bar\lambda(\bar x)\right)\,
         d\bar\lambda(\bar s)=\bar\lambda(\partial U(r^{-n}))\label{ZZZ3b},
\end{eqnarray}
using again the properties of $\chi_n$. Now, (\ref{ZZZ3a}) and (\ref{ZZZ3b}) directly give (\ref{ZZZ3}).
Due to (\ref{ZZZ1}), (\ref{ZZZ2}) and (\ref{ZZZ3}), we have
\begin{multline*}
\frac 1{\bar\lambda(U)}
  \left|\int_{\mathbb T^d} {\mathbf 1}_U\circ T^{-n}.\varphi\, d\bar\lambda\right|\le
  \frac 1{\bar\lambda(U)}\Big(
     \Big|\int_{\mathbb T^d} (\chi_n*{\mathbf 1}_U)\circ T^{-n}.(\chi_n*\varphi)) \, d\bar\lambda\Big|\\
+\bar\lambda(U)\omega(\varphi,r^{-n})+3\Vert \varphi \Vert_\infty
    \bar\lambda(\partial U(r^{-n}))\Big).
\end{multline*}
Now, the hypothesis on $m_u(\partial\mathcal C(\beta))$ implies that there exists $Q_1$ (depending on $Q$ and on
$T$) such that
$$\forall n\ge 0,\ \ \  \bar\lambda(\partial U(r^{-n}))\le Q_1 r^{-n} \, . $$
Moreover, applying Proposition \ref{decorr} with $f=\chi_n*\varphi$ and $g=\chi_n*\mathbf 1_U$ and using the following
facts
$$\Vert\chi_n*\varphi\Vert_\infty\le\Vert\varphi\Vert_\infty,\ \
    \Vert\chi_n*\mathbf 1_U\Vert_\infty\le 1,\ Lip(\chi_n*\mathbf 1_U)\le Lip(\chi_n)\ \ \mbox{and}\ \
   Lip(\chi_n*\varphi)\le \Vert\varphi\Vert_\infty Lip(\chi_n),$$
we get the existence of $\tilde C_0$ (depending on $C_0$ and on $Q$) such that we have
\begin{eqnarray*}
\frac 1{\bar\lambda(U)}
  \left|\int_{\mathbb T^d} {\mathbf 1}_U\circ T^{-n}.\varphi\, d\bar\lambda\right|
&\le& \tilde C_0\Vert \varphi \Vert_\infty\frac{r^{-n}+(1+r^{n(d+1)})\xi_0^n}{\varepsilon_n^{d_e+d_s}
     m_u({\mathcal C})}+ \omega(\varphi,r^{-n})\\
&\le& 3\tilde C_0\Vert \varphi \Vert_\infty\frac{\xi_0^{n/(d+2)}}{\varepsilon_n^{d_e+d_s}
     m_u({\mathcal C})}+ \omega(\varphi,\zeta_1^n),
\end{eqnarray*}
since $r^{-1}=r^{d+1}\xi_0=\xi_0^{1/(d+2)}$.
We conclude by taking $\xi_1:=\xi_0^{1/(d+2)} \alpha^{-(d_e+d_s)}<1$.
\end{proof}

In the next result (which is an adaptation of Proposition 1.4 of \cite{SLBFP}),
we prove that Proposition \ref{PROP1} holds true with the stable-neutral continuity modulus
$\omega_{(s,e)}$ instead of $\omega$.
\begin{prop}\label{PROP2}
Let $\zeta_1\in (\xi_0^{1/((d+2)(d_e+d_s))},1)$  where $\xi_0$ is given in Proposition \ref{decorr}.
There exist $C_2>0$, $N_2\ge 1$ and $\xi_2\in(0,1)$ such that,
for every $\bar\lambda$-centered bounded function
$\varphi:\mathbb T^d\rightarrow\mathbb R$, every $\bar x\in\mathbb T^d$, every $n\ge N_2$
and every bounded convex set
$\mathcal C\subseteq E_u$ with diameter smaller than $r_0$ and satisfying
$m_u(\partial \mathcal C(\beta))\le Q\gamma$, we have
$$\left|\frac 1{m_u(S^n(\mathcal C))}\int_{S^n{\mathcal C}}\varphi(\bar x+\overline{h_u})\, dm_u(h_u)\right|
      \le K_2\left( \frac{\Vert \varphi \Vert_\infty}{m_u({\mathcal C})}\xi_2^n
            +\omega_{(s,e)}(\varphi,\zeta_1^n)\right) \, . $$
\end{prop}

\begin{proof}
We consider a finite cover of $\mathbb T^d$ by sets $P_i=\bar y_i+\overline{B_u(r_0)+B_e(r_0)
+B_s(r_0)}$ for $i=1,...,I$, $\bar y_i$ being fixed points of $\mathbb T^d$.
We consider a partition of the unity $H_1,...,H_I$ (i.e. $\sum_{i=1}^IH_i=1$) such that
each $H_i$ is infinitely differentiable, with support in $P_i$.
Let $\varphi:\mathbb T^d\rightarrow\mathbb R$ be a bounded centered function.
For every $i=1,...,I$, we define $\varphi_i:=H_i\varphi$.
We have
\begin{equation}\label{EQ1}
 \int_{S^n{\mathcal C}}\varphi(\bar x+\overline{h_u})\, dm_u(h_u)
      = \sum_{i=1}^I\int_{S^n{\mathcal C}}\varphi_i(\bar x+\overline{h_u})\, dm_u(h_u).
\end{equation}
We also consider a continuously differentiable function $g:E_u\rightarrow [0,+\infty)$
with support in $B_u(r_0)$ and such that $\int_{E_u}g(h_u)\, dm_u(h_u)=1$.
We approximate now each $\varphi_i$ by a regular function $\psi_i$ by setting,
for every $(h_u,h_e,h_s)\in B_u(r_0)\times B_e(r_0)\times B_s(r_0)$,
$$\psi_i(\bar y_i+\overline{h_u}+\overline{h_e}+\overline{h_s})=
     g(h_u)\int_{B_u(r_0)}\varphi_i(\bar y_i+\overline{h'_u}+\overline{h_e}+\overline{h_s})
    \, dm_u(h'_u),$$
$\psi_i$ being null outside of $P_i$.
We observe that
$$\int_{P_i}\psi_i\, d\bar\lambda= \int_{P_i}\varphi_i\, d\bar\lambda,$$
that $||\psi_i||_\infty\le \Vert \varphi \Vert_\infty\Vert g \Vert_\infty m_u(B_u(r_0))$
and that, for every $\delta>0$,
\begin{eqnarray*}
\omega(\psi_i,\delta)&\le& m_u(B_u(r_0))\left[
    \Vert \varphi \Vert_\infty Lip(g)\delta+\Vert g \Vert_\infty\omega_{(s,e)}(\varphi_i,\delta)\right]\\
  &\le& m_u(B_u(r_0))\left[
    \Vert \varphi \Vert_\infty Lip(g)\delta+\Vert g \Vert_\infty\Vert \varphi \Vert_\infty Lip(H_i)\delta
      +\Vert g \Vert_\infty\omega_{(s,e)}(\varphi,\delta) \Vert H_i \Vert_\infty\right].
\end{eqnarray*}
Now, applying Proposition \ref{PROP1} to $\psi_i$, for every $n\ge N_1$, we have
\begin{equation}\label{EQ2}
\left|\frac 1{m_u(S^n{\mathcal C})}\int_{S^n{\mathcal C}}\psi_i(\bar x+\overline{h_u})\, dm_u(h_u)\right|
 \le K'_1\left(\frac{ \Vert \varphi \Vert_\infty\xi_1^n}{m_u({\mathcal C})}+\omega_{(s,e)}(\varphi,\zeta_1^n)
       +\Vert \varphi \Vert_\infty\zeta_1^n\right).
\end{equation}
We observe that the connected components of $(\bar x+\overline{S^n{\mathcal C}})\cap P_i$ are
$\bar x+\overline{C_{i,j}}$, where $C_{i,j}$ are some connected subsets of $E_u$.
We have
$$\int_{S^n{\mathcal C}}\varphi_i(\bar x+\overline{h_u})\, dm_u(h_u)=\sum_j\int_{C_{i,j}}
\varphi_i(\bar x+\overline{h_u})\, dm_u(h_u)$$
and
$$\int_{S^n{\mathcal C}}\psi_i(\bar x+\overline{h_u})\, dm_u(h_u)=\sum_j\int_{C_{i,j}}
\psi_i(\bar x+\overline{h_u})\, dm_u(h_u) \, . $$
Now, if $C_{i,j}$ does not contain any point of $\partial(S^n{\mathcal C})$, then there exists
$h_e^{(j)}\in B_e(r_0)$ and $h_s^{(j)}\in B_s(r_0)$ such that
$$\bar x+\overline{C_{i,j}}=
   \left\{\bar y_i+\overline{h_e^{(j)}}+\overline{h_s^{(j)}}+\overline{h_u};\
   \ h_u\in B_u(r_0)\right\} \, . $$
Using the definition of $\psi_i$, we get
\begin{eqnarray*}
\int_{C_{i,j}} \psi_i(\bar x+\overline{h_u})\, dm_u(h_u)&=&\int_{B_u(r_0)}
     \psi_i(\bar y_i+\overline{h_e^{(j)}}+\overline{h_s^{(j)}}+\overline{h_u})\, dm_u(h_u)\\
  &=& \int_{B_u(r_0)}
     \varphi_i(\bar y_i+\overline{h_e^{(j)}}+\overline{h_s^{(j)}}+\overline{h_u})\, dm_u(h_u),
\end{eqnarray*}
since $\int_{B_u(r_0)}g(h_u)\, dm_u(h_u)=1$ and so
$$
\int_{C_{i,j}} \psi_i(\bar x+\overline{h_u})\, dm_u(h_u)
= \int_{C_{i,j}} \varphi_i(\bar x+\overline{h_u})\, dm_u(h_u).
$$
Therefore we have
\begin{eqnarray}
\left|\frac 1{m_u(S^n{\mathcal C})}\int_{S^n{\mathcal C}}(\psi_i(\bar x+\overline{h_u})
    -\varphi_i(\bar x+\overline{h_u}))\,
     dm_u(h_u)\right|
 &\le& 2\Vert \varphi \Vert_\infty\frac{m_u(\partial(S^n {\mathcal C})(r_0))}{m_u(S^n{\mathcal C})}\nonumber\\
 &\le& 2\Vert \varphi \Vert_\infty\frac{m_u(\partial {\mathcal C} (K\rho_u^n r_0))}{m_u({\mathcal C})}\nonumber\\
 &\le& 2\Vert \varphi \Vert_\infty\frac{Q K\rho_u^n r_0}{m_u({\mathcal C})}\, .
 \label{EQ3}
\end{eqnarray}
We conclude thanks to (\ref{EQ1}), (\ref{EQ2}) and (\ref{EQ3}),
by taking $\xi_2:=\max(\xi_1,\zeta_1,\rho_u)$.
\end{proof}
\begin{proof}[Proof of Theorem \ref{THM2}]

We start by proving the first point. By  Proposition \ref{partition},
\beq \label{THM2pritem1}
{\mathbb E}[\varphi|\mathcal F_n](\bar x) - \varphi(\bar x)  =\frac 1{m_u(S^{-n}F_{T^n\bar x})}
                 \int_{S^{-n}F_{T^n\bar x}} \big ( \varphi(\bar x+\overline{h_u}) - \varphi(\bar x) \big ) \, dm_u(h_u) \, .
\eeq
Let $h_u \in S^{-n}F_{T^n\bar x}$ and $y \in F_{T^n\bar x}$ such that $ h_u = S^{-n} (y) $. Take now $\beta_u \in (r_u , \rho_u)$. From (\ref{hyp-a}) and the fact that $F_{T^n\bar x}$ is uniformly bounded, we derive that there exists a positive constant $C$ such that $\Vert h_u \Vert \leq C \beta_u^n$. 
Therefore, starting from \eqref{THM2pritem1}, by definition of $\omega_{(u)} ( \varphi , \delta)$, we get 
$$
\Vert {\mathbb E}[\varphi|\mathcal F_n] - \varphi  \Vert_{\infty} \leq \omega_{(u)} ( \varphi , C \beta_u^n) \, .
$$
The first point of Theorem \ref{THM2} then comes from the fact that there exists $N>0$ such that for any $n \geq N$, $C \beta_u^n \leq \rho_u^n$. 

\smallskip

We turn now to the proof of the second point. Let $\zeta_1$, $C_2$, $\xi_2$ and $N_2$ as in Proposition \ref{PROP2} with $\zeta_1<\zeta$.
Let $\beta\in(\xi_2,1)$ and $\mathcal V_n:=\{m_u(F_{\cdot})\ge \beta^n\}$.
We take $\xi=\max(\xi_2/\beta,\beta^{\frac 1 {d_u}})$.
To prove the second point, we use again the expression
of ${\mathbb E}[\varphi|\mathcal F_{-n}]$ given in Proposition \ref{partition} and
we apply Proposition \ref{PROP2} with $\mathcal C=F_{T^{-n}(\bar x)}$ with the notation of Proposition
\ref{partition}.

\smallskip

It remains to prove the last point of the theorem. It comes from the fact (proved in Proposition II.1 of \cite{SLB}) that
$$\exists L>0,\ \ \forall n\ge 0,\ \bar\lambda(m_u(F_{\cdot})< \beta^n )\le L\beta^{\frac n {d_u}}  \, . $$
\end{proof}

\section{Proof of Theorems \ref{EmpLp} and \ref{Kiefer}}\label{lastsec}
\setcounter{equation}{0}
In this section, $C$ is a positive constant which may vary from lines to lines, and
the notation
 $a_n \ll b_n$ means that there exists a numerical constant $C$ not
depending on $n$ such that  $a_n \leq  Cb_n$, for all positive integers $n$.

\begin{proof}[Proof of Theorem \ref{EmpLp}]
The proof is based on Proposition \ref{PIFBanach}
of Section \ref{probresults}, which gives sufficient
conditions for the weak invariance principle in 2-smooth Banach spaces.

Let $Y_i(s)= {\bf 1}_{f \circ T^i \leq s}-F(s)$
and let ${\mathcal F}_i$ be the filtration introduced in Section \ref{ineqauto}. 
Note first that, for $ 2\leq p < \infty$, the space
${\mathbb L}^p$ is 2-smooth and $p$-convex (see \cite{Pi}). Moreover it has a Schauder basis (and even an
unconditional basis).

Hence it suffices to check (\ref{condp0}) of Proposition \ref{PIFBanach}.
According to Lemma 6.1 of \cite{DeMePe} (with $b_k=1$), there exists a positive constant $C$ such that
\begin{align*}
  \sum_{k=1}^\infty \| \|P_{-k}(Y_0)\|_{{\mathbb L}^p}\|_2
  &\leq C\sum_{k=1}^\infty \Big(\frac{1}{k} \sum_{i=k}^\infty
  \| \|P_{-i}(Y_0)\|_{{\mathbb L}^p}\|_2^p\Big)^{1/p}\leq
  C\sum_{k=1}^\infty \Big(\frac{1}{k} \sum_{i=k}^\infty
  \| \|P_{-i}(Y_0)\|_{{\mathbb L}^p}\|_p^p\Big)^{1/p} \, ,\\
  \text{and}
  \sum_{k=-\infty}^0 \| \|P_{-k}(Y_0)\|_{{\mathbb L}^p}\|_2 &
  \leq C\sum_{k=1}^\infty \Big(\frac{1}{k} \sum_{i=k}^\infty
  \| \|P_{i+1}(Y_0)\|_{{\mathbb L}^p}\|_2^p\Big)^{1/p}\leq
  C\sum_{k=1}^\infty \Big(\frac{1}{k} \sum_{i=k}^\infty
  \| \|P_{i+1}(Y_0)\|_{{\mathbb L}^p}\|_p^p\Big)^{1/p} \, .
\end{align*}
Since ${\mathbb L}^p$ is $p$-convex, it follows that
$$
\sum_{i=k}^\infty \| \|P_{-i}(Y_0)\|_{{\mathbb L}^p}\|_p^p \leq
 K \| \|{\mathbb E}(Y_k|{\mathcal F}_{0})\|_{{\mathbb L}^p}\|_p^p\quad
 \text{and} \quad
\sum_{i=k}^\infty
  \| \|P_{i+1}(Y_0)\|_{{\mathbb L}^p}\|_p^p \leq
  K \| \|Y_{-k}-{\mathbb E}(Y_{-k}|{\mathcal F}_{0})\|_{{\mathbb L}^p}\|_p \, ,
$$
for some positive constant $K$. 
Hence (\ref{condp0}) is true as soon as
$$
  \sum_{n\geq 1} \frac{1}{n^{1/p}}
  \| \|{\mathbb E}(Y_n|{\mathcal F}_{0})\|_{{\mathbb L}^p}\|_p < \infty
  \quad \text{and} \quad
  \sum_{n\geq 1} \frac{1}{n^{1/p}}
  \| \|Y_{-n}-{\mathbb E}(Y_{-n}|{\mathcal F}_{0})\|_{{\mathbb L}^p}\|_p
  <\infty \, .
$$
Let us have a look to
\begin{align*}
\| \|{\mathbb E}(Y_n|{\mathcal F}_{0})\|_{{\mathbb L}^p}\|_p
 & = \Big( {\mathbb E} \int_{\mathbb R} |F_{f \circ T^n|{\mathcal F}_0}(t)-F(t)|^p dt \Big)^{1/p}\\
 & \leq \Big( {\mathbb E} \int_{\mathbb R}  |F_{f \circ T^n|{\mathcal F}_0}(t)-F(t)| dt \Big)^{1/p}
 \, ,
\end{align*}
with $F_{f \circ T^n|{\mathcal F}_0}(t):={\mathbb P}(f \circ T^n\le t|{\mathcal F}_0)$.
Now
$$
\int_{\mathbb R}  |F_{f \circ T^n|{\mathcal F}_0}(t)-F(t)| dt=
\sup_{g \in \Lambda_1}
\Big |{\mathbb E}(g\circ f\circ T^n|{\mathcal F}_0)-{\mathbb E}(g \circ f)\Big |\, ,
$$
where $\Lambda_1$ is the set of $1$-lipschitz functions.
Hence, since $\omega_{(s,e)}(g\circ f,\cdot)$ is smaller than $\omega_{(s,e)}(f, \cdot)$,
it follows from (\ref{THM2b}) and (\ref{THM2c}) of Theorem \ref{THM2} that
$$
\| \|{\mathbb E}(Y_n|{\mathcal F}_{0})\|_{{\mathbb L}^p}\|_p
 \leq \Big( {\mathbb E}\Big(
\sup_{g \in \Lambda_1}
\Big |{\mathbb E}(g\circ f\circ T^n|{\mathcal F}_0)-{\mathbb E}(g \circ f)\Big | \Big )
\Big)^{1/p}
\leq
C((\omega_{(s,e)}(f,\zeta^n))^{1/p}
      + \Vert f\Vert_\infty^{1/p}\xi^{n/p})\, ,
$$
by noticing that we can replace $\Lambda_1$ by the set of $g\in\Lambda_1$ such that $g\circ f(0)=0$. 
In the same way, due to (\ref{THM2a}) of Theorem \ref{THM2}, we have
$$
\| \|Y_{-n}-{\mathbb E}(Y_{-n}|{\mathcal F}_{0})\|_{{\mathbb L}^p}\|_p
\leq C(\omega_{(u)}
     (f,\rho_u^n))^{1/p}\, .
$$
The result follows.
\end{proof}

\begin{proof}[Proof of Theorem \ref{Kiefer}]
Our aim is to apply the tightness criterion  given in Proposition \ref{p2}.
Let $X_i=f \circ T^i$ and let ${\mathcal F}_i$ be the filtration 
defined in Section \ref{ineqauto}. We  need the following upper bounds.  

\begin{lem} \label{lmamajcov}
Let $g_{s,t}(v)={\bf 1}_{v \leq t}-{\bf 1}_{v \leq s}$, and let $P$ be the image
measure of $\bar \lambda$ by $f$.
Under the assumptions of Theorem \ref{Kiefer}, we have, for any $\beta>1$,
$$
  \sum_{k=0}^n |{\mathrm{Cov}}(g_{s,t}(X_0), g_{s,t}(X_k))|\ll \|g_{s,t}\|_{P,1}^{(\beta+ \alpha -1)/(\beta+\alpha)}\sum_{k=0}^n \frac{1}{(k+1)^{a \alpha/(\beta+\alpha)}} \, .
$$
\end{lem}

\begin{lem} \label{lmacalculcoeff}
Under the assumptions of Theorem \ref{Kiefer}, we have, for any $p \geq 1$,
\begin{align*}
\Vert {\mathbb{E}}%
_{0}  (  g_{s,t}(X_{k})) - {\mathbb E} (g_{s,t} (X_k)  ) \Vert_{p} &\ll k^{- a \alpha  / ( \alpha + p )}\\
\Vert g_{s,t}(X_{0}) - {\mathbb{E}}%
_{k}( g_{s,t}(X_{0}))\Vert_{p} &\ll k^{- a \alpha  / ( \alpha + p )}\, ,
\end{align*}
and, for any $p \geq 2$, 
$$
A(g_{s,t}(X)-{\mathbb E}(g_{s,t}(X)), j) \ll j^{-2 a  \alpha  / (2 \alpha + p )} \, ,
$$
where the coefficient $A(g_{s,t}(X)-{\mathbb E}[g_{s,t}(X)], j)$ is defined in (\ref{notalambda}). The constants involved in the symbol $\ll$ do not depend on $(s,t)$. 
\end{lem}

Let us continue the proof of Theorem \ref{Kiefer} with the help
of these lemmas.
From Proposition \ref{consdirect}, Lemma \ref{lmamajcov} and Lemma \ref{lmacalculcoeff}, we derive that, for  $p>2$,
\begin{multline*}
\Big \| \max_{1\leq k\leq n} |S_{k} (g_{s,t}) | \Big \|_{p}\ll n^{1/2}
\Big ( \|g_{s,t}\|_{P,1}^{(\beta+ \alpha -1)/(\beta+\alpha)}\sum_{k=1}^n \frac{1}{k^{a \alpha/(\beta+\alpha)}} \Big )^{1/2}%
\,\\
+n^{1/p}\sum_{k=1}^{2n} \frac{k^{-a\alpha/(\alpha+p)}}{k^{1/p}}
+ n^{1/p}\Big (\sum_{k=1}^{n}\frac{k^{-2 a  \alpha  / (2 \alpha + p )}}{k^{(2/p)-1}}(\log k)^{\gamma
}\Big )^{1/2},
\end{multline*}
where $\gamma$ can be taken $\gamma=0$ for $2<p\leq3$ and $\gamma>p-3$ for
$p>3$. Therefore if
$$
a > \max \Big ( 1 + \frac{\beta}{\alpha} , \frac{(p-1) ( 2 \alpha + p)} {p \alpha} \Big ) \, ,
$$
then setting $r= 2 ( \beta + \alpha) /(\beta + \alpha -1)$, we  get that
$$
\Big \| \max_{1\leq k\leq n} |S_{k} (g_{s,t}) | \Big \Vert_{p} \ll n^{1/2}
 \|g_{s,t}\|_{P,1}^{1/r}
+ n^{1/p} \, .
$$
We shall apply the tightness criterion given in Proposition \ref{p2}.
Since ${\mathcal N}_{P,1}(x, {\mathcal F}) \leq C x^{-\ell}$ for the class ${\mathcal F} = \{u \mapsto {\bf 1}_{u \leq t} , t \in {\mathbb R}^\ell  \}$, we get 
\begin{equation}\label{ent1}
\int_0^1 x^{(1-r)/r} ({\mathcal N}_{P,1}(x, {\mathcal F}))^{1/p} dx \leq C  \int_0^1 x^{(1-r)/r} x^{-\ell/p} dx < \infty ,
\end{equation}
 as soon as $ p > 2 \ell ( \beta+ \alpha) /( \beta + \alpha -1)$. Moreover
\begin{equation}\label{ent2}
\lim_{x \rightarrow 0} x^{p-2}  {\mathcal N}_{P,1}(x, {\mathcal F}) =0
\end{equation}
as soon as $ p >2 + \ell $.

Hence if $p \in ]2,  2 \ell ( 1 + \alpha^{-1})] $, we take $\beta=(2\alpha \ell + (1-\alpha)p)/(p-2\ell)+ \varepsilon$ for some positive 
and small enough $\varepsilon$ (so that $\beta>1$), and
we infer that (\ref{ent1}) and (\ref{ent2}) hold provided that $p>\max(\ell+2, 2\ell)$ and
$$
a > k_{\ell, \alpha}(p)=\max \Big (\frac{p}{\alpha (p-2\ell)} , \frac{(p-1) ( 2 \alpha + p)} {p \alpha} \Big ) \, .
$$
Taking the minimum in $p\geq \max(\ell+2, 2\ell)$ on the right hand, we obtain
that (\ref{ent1}) and (\ref{ent2}) hold provided that $a>a(\ell, \alpha)$, where
$a(\ell, \alpha)$ has been defined in (\ref{defaellalpha}). 

We infer that the conditions (\ref{equiibis}) and (\ref{equiiter})
of Proposition \ref{p2} hold for this choice of $a$, which proves the tightness
of the empirical process (see \cite{VW},  page 227).

 To prove the weak convergence of the finite
dimensional distribution, it suffices to show that for any $(\alpha_1, \dots, \alpha_m) \in {\mathbb R}^m$ and any $(s_1, \dots, s_m) \in ({\mathbb R}^{\ell} )^m$, the process 
$$
\Big \{ n^{-1/2}\sum_{i=1}^m \alpha_i S_{[nt]} (s_i) \, , \, t \in [0,1] \Big \} 
\quad 
\text{converges in distribution in $D_{\mathbb R}([0,1])$ to $W$,} 
$$
where $W$ is a Wiener process such that ${\rm Cov} (W_{t_1}, W_{t_2}) = 
\min ( t_1, t_2 ) \sum_{i=1}^m \sum_{j=1}^m \alpha_i \alpha_j \Lambda (s_i, s_j)$. Note that $\sum_{i=1}^m \alpha_i S_{[nt]} (s_i)=\sum_{k=1}^{[nt]} Y_k$ where $Y_k = \sum_{i=1}^m \alpha_i  \big ( {\bf 1}_{ X_k \leq s_i  } -F(s_i) \big ) $. Therefore, the above convergence in distribution will follow from Proposition 5 in \cite{DeMeVo} if we can prove that 
\beq \label{condipnonadap}
\sum_{k=1}^{\infty} \frac{\Vert {\mathbb E}_0 ( Y_k ) \Vert_2}{\sqrt k } < \infty  \, \text{ and } \, \sum_{k=1}^{\infty} \frac{\Vert Y_{0} -  {\mathbb E}_k ( Y_{0} ) \Vert_2}{\sqrt k } < \infty \, .
\eeq
By the triangle inequality, it suffices to prove that \eqref{condipnonadap} holds with $ {\bf 1}_{ X_k \leq s  } -F(s) $ in place of $Y_k$. This follows from Lemma \ref{lmacalculcoeff} as soon as $a>(\alpha+2)/2 \alpha$. 
\end{proof}

\begin{proof}[Proof of Lemma \ref{lmamajcov}] We prove the results for $\ell=2$. The general case can be proved in the same way.
For $u \in {\mathbb R}$, let
$h_u(x)={\bf 1}_{x \leq u}$. By definition of $g_{s,t}$,
$$
g_{s,t}=h_{t_1}\otimes h_{t_2} - h_{s_1}\otimes h_{s_2}\, ,
$$
with the notation $(G_1\otimes G_2)(u_1,u_2):=G_1(u_1)G_2(u_2)$.
For $\varepsilon>0$, let
$$h_{u, \varepsilon}(x)={\bf 1}_{x \leq u} -
\varepsilon^{-1}(x-u-\varepsilon){\bf 1}_{u<x \leq u+\varepsilon}\, ,$$
and note that $h_{u, \varepsilon}$ is Lipschitz with Lipschitz constant $\varepsilon^{-1}$.
We have the decomposition
$ h_{t_1}\otimes h_{t_2}=h_{t_1, \varepsilon}\otimes h_{t_2, \varepsilon}
+ R_{t, \varepsilon}$, where
$$R_{t, \varepsilon}= (h_{t_1}-h_{t_1,\varepsilon})\otimes h_{t_2} +
h_{t_1, \varepsilon}\otimes (h_{t_2}-h_{t_2, \varepsilon})\, .
$$
Setting
$$
g_{s,t, \varepsilon}=h_{t_1, \varepsilon}\otimes h_{t_2, \varepsilon} -
 h_{s_1, \varepsilon}\otimes h_{s_2, \varepsilon}\, ,
$$
we obtain the decomposition
\begin{equation}\label{dec}
g_{s,t}=g_{s,t, \varepsilon} + H_{s,t, \varepsilon}, \quad \text{with} \quad  H_{s,t, \varepsilon}=
R_{t, \varepsilon}- R_{s, \varepsilon}\, .
\end{equation}
On the other hand, we have 
$$
\mathrm {Cov}(g_{s,t}(X_0), g_{s,t}(X_k))
=
{\mathbb E}((g_{s,t}(X_0)-
{\mathbb E}(g_{s,t}(X_0)|{\mathcal F}_{[k/2]}))
g_{s,t}(X_k))
+
\mathrm {Cov}({\mathbb E}(g_{s,t}(X_0)|{\mathcal F}_{[k/2]}), g_{s,t}(X_k))\, .
$$
Using (\ref{dec}), we get
\begin{multline}\label{zero}
{\mathbb E}((g_{s,t}(X_0)-
{\mathbb E}(g_{s,t}(X_0)|{\mathcal F}_{[k/2]}))
g_{s,t}(X_k))={\mathbb E}((g_{s,t, \varepsilon}(X_0)-
{\mathbb E}(g_{s,t, \varepsilon}(X_0)|{\mathcal F}_{[k/2]}))
g_{s,t}(X_k)) \\
+{\mathbb E}((H_{s,t, \varepsilon}(X_0)-
{\mathbb E}(H_{s,t, \varepsilon}(X_0)|{\mathcal F}_{[k/2]}))
g_{s,t}(X_k))\, .
\end{multline}
Applying (\ref{THM2a}) of Theorem \ref{THM2}, we infer that
\begin{equation}\label{un}
|{\mathbb E}((g_{s,t, \varepsilon}(X_0)-
{\mathbb E}(g_{s,t, \varepsilon}(X_0)|{\mathcal F}_{[k/2]}))
g_{s,t}(X_k))| \leq C \|g_{s,t}\|_{P,1} 
 \varepsilon^{-1} \omega_{(u)}(f,\rho_{u}^{[k/2]}) \, .
\end{equation}
Applying H\"older's inequality, and using the fact that the distributions functions of
$f_1$ and $f_2$ are H\"older continuous of order $\alpha$, we get 
\begin{equation}\label{deux}
|{\mathbb E}((H_{s,t, \varepsilon}(X_0)-
{\mathbb E}(H_{s,t, \varepsilon}(X_0)|{\mathcal F}_{[k/2]}))
g_{s,t}(X_k))| \leq C\|g_{s,t}\|_{P, 1}^{(\beta-1)/\beta} \varepsilon^{\alpha/\beta} \, .
\end{equation}
Using (\ref{dec}) again, we also have 
\begin{multline}\label{trois}
\mathrm {Cov}({\mathbb E}(g_{s,t}(X_0)|{\mathcal F}_{[k/2]}), g_{s,t}(X_k))
=\mathrm {Cov}({\mathbb E}(g_{s,t}(X_0)|{\mathcal F}_{[k/2]}), g_{s,t, \varepsilon}(X_k))
\\
+\mathrm {Cov}({\mathbb E}(g_{s,t}(X_0)|{\mathcal F}_{[k/2]}), H_{s,t, \varepsilon}(X_k))
\, .
\end{multline}
To handle the first term in the right-hand side, we set $g^{(0)}_{s,t, \varepsilon}(X_0) = g_{s,t, \varepsilon}(X_0)- \E (g_{s,t, \varepsilon}(X_0))$ and  note first that 
\begin{align*}
\mathrm {Cov}({\mathbb E}(g_{s,t}(X_0)|{\mathcal F}_{[k/2]}), g_{s,t, \varepsilon}(X_k)) &= \E ({\mathbb E}(g_{s,t}(X_{-k})|{\mathcal F}_{[k/2]-k})   g^{(0)}_{s,t, \varepsilon}(X_0)  ) \\
& =\E (g_{s,t}(X_{-k}) {\mathbb E}(  g^{(0)}_{s,t, \varepsilon}(X_0) |{\mathcal F}_{[k/2]-k} ))\, .
\end{align*}
Therefore, considering the set ${\mathcal V}_{n}$ introduced in Theorem \ref{THM2}, it follows that 
\begin{multline*}
|\mathrm {Cov}({\mathbb E}(g_{s,t}(X_0)|{\mathcal F}_{[k/2]}), g_{s,t, \varepsilon}(X_k))|\leq 
 2 \Vert g_{s,t, \varepsilon}(X_0)  \Vert_{\infty} \E ( |g_{s,t}(X_{-k}) | {\bf 1}_{{\mathcal V}_{k-[k/2]}^c})   \\
 +  \E ( | g_{s,t}(X_{-k}) |  |{\mathbb E}(  g_{s,t, \varepsilon}(X_0) |{\mathcal F}_{[k/2]-k} ) - \E (  g_{s,t, \varepsilon}(X_0) ) | {\bf 1}_{{\mathcal V}_{k-[k/2]}}) \, .
\end{multline*}
On one hand, applying (\ref{THM2b}) of Theorem \ref{THM2} with $\varphi=  g_{s,t, \varepsilon} \circ f$ and using the fact that, since $h_{u,\varepsilon}$ is Lipschitz with Lipschitz constant $\varepsilon^{-1}$, $\omega_{(s,e)}(g_{s,t, \varepsilon} \circ f,\zeta^{[k/2]}) \leq 4 \varepsilon^{-1} \omega_{(s,e)}(f,\zeta^{[k/2]})$, we infer that
\begin{multline*}
 \E ( | g_{s,t}(X_{-k}) |  |{\mathbb E}(  g_{s,t, \varepsilon}(X_0) |{\mathcal F}_{[k/2]-k} ) - \E (  g_{s,t, \varepsilon}(X_0) ) | {\bf 1}_{{\mathcal V}_{k-[k/2]}}) \\
 \leq 
C \|g_{s,t}\|_{P,1}(\xi^{[k/2]}+   \varepsilon^{-1} \omega_{(s,e)}(f,\zeta^{[k/2]}))\, .
\end{multline*}
On the other hand, since $\bar\lambda({\mathcal V}_{k-[k/2]}^c)\leq C \xi^{[k/2]}$,
applying H\"older's inequality, we get 
\begin{equation*}\label{quatre}
{\mathbb E}(|(g_{s,t}(X_{-k})|
{\bf 1}_{{\mathcal V}_{k-[k/2]}^c})
\leq
C \|g_{s,t}\|_{P,1}^{(\beta+\alpha-1)/(\beta+\alpha)} \xi^{[k/2]/(\beta+\alpha)} \, .
\end{equation*}
So, overall, 
\begin{multline}\label{troisbis}
|\mathrm {Cov}({\mathbb E}(g_{s,t}(X_0)|{\mathcal F}_{[k/2]}), g_{s,t, \varepsilon}(X_k))|\leq
C \|g_{s,t}\|_{P,1}^{(\beta+\alpha-1)/(\beta+\alpha)} \xi^{[k/2]/(\beta+\alpha)} \\
+
C \|g_{s,t}\|_{P,1}(\xi^{[k/2]}+   \varepsilon^{-1} \omega_{(s,e)}(f,\zeta^{[k/2]}))\, .
\end{multline}
We handle now the second term in the right-hand side of \eqref{trois}. Applying H\"older's inequality again, and using that the distributions functions of
$f_1$ and $f_2$ are H\"older continuous of order $\alpha$, we get 
\begin{equation}\label{cinq}
|\mathrm {Cov}({\mathbb E}(g_{s,t}(X_0)|{\mathcal F}_{[k/2]}), H_{s,t, \varepsilon}(X_k))|
\leq C\|g_{s,t}\|_{P, 1}^{(\beta-1)/\beta} \varepsilon^{\alpha/\beta} \, .
\end{equation}
Therefore, starting from \eqref{trois} and considering \eqref{troisbis} and \eqref{cinq}, it follows that
\begin{multline}\label{troister}
|\mathrm {Cov}({\mathbb E}(g_{s,t}(X_0)|{\mathcal F}_{[k/2]}), g_{s,t}(X_k)) | \leq C \|g_{s,t}\|_{P,1}^{(\beta+\alpha-1)/(\beta+\alpha)} \xi^{[k/2]/(\beta+\alpha)} \\
+
C \|g_{s,t}\|_{P,1}(\xi^{[k/2]}+   \varepsilon^{-1} \omega_{(s,e)}(f,\zeta^{[k/2]}))C\|g_{s,t}\|_{P, 1}^{(\beta-1)/\beta} \varepsilon^{\alpha/\beta} \, .
\end{multline}

Gathering the  bounds (\ref{zero}), (\ref{un}), (\ref{deux}) and
(\ref{troister}), it follows that
$$
|\mathrm {Cov}(g_{s,t}(X_0), g_{s,t}(X_k))|
\leq C\Big(\|g_{s,t}\|_{P, 1}\frac{1}{\varepsilon k^a}+ \|g_{s,t}\|_{P, 1}^{(\beta-1)/\beta} \varepsilon^{\alpha/\beta} +
\|g_{s,t}\|_{P,1}^{(\beta+\alpha-1)/(\beta+\alpha)} \xi^{[k/2]/(\beta+\alpha)}
\Big)
\, .
$$
Taking $\varepsilon=\|g_{s,t}\|_{P, 1}^{1/(\alpha+ \beta)}k^{-a\beta/(\alpha+ \beta)}$, we get 
$$
|\mathrm {Cov}(g_{s,t}(X_0), g_{s,t}(X_k))|
\leq
C\|g_{s,t}\|_{P, 1}^{(\beta+\alpha-1)/(\beta+\alpha)}\Big(\frac{1}{k^{a\alpha/(\alpha + \beta)}}  + \xi^{[k/2]/(\beta+\alpha)}
\Big)\, .
$$
The result follows by summing in $k$.
\end{proof}

\begin{proof}[Proof of Lemma \ref{lmacalculcoeff}]
Using the same notations as in the proof of Lemma \ref{lmamajcov}, and using that the distribution
functions of $f_1$ and $f_2$ are H\"older continuous of order $\alpha$, we obtain 
$$
\Vert {\mathbb{E}}
_{0}  (  g_{s,t}(X_{k})) - {\mathbb E} (g_{s,t} (X_k)  ) \Vert_{p}\leq
\Vert {\mathbb{E}}
_{0}  (  g_{s,t, \varepsilon}(X_{k})) -
{\mathbb E} (g_{s,t, \varepsilon} (X_k)  ) \Vert_{p}
+ C \varepsilon^{\alpha/p}\, .
$$
Recall that the set ${\mathcal V}_{n}$ introduced in Theorem \ref{THM2} is such that
$\bar \lambda({\mathcal V}_n^c)\leq C \xi^n$.
Applying Theorem \ref{THM2} (see (\ref{THM2d})), we obtain 
$$
\Vert {\mathbb{E}}
_{0}  (  g_{s,t, \varepsilon}(X_{k})) -
{\mathbb E} (g_{s,t, \varepsilon} (X_k)  ) \Vert_{p}
\leq C (\varepsilon ^{-1} \omega_{(s,e)}(f,\zeta^k)+ \xi^{k/p})\, .
$$
Consequently
$$
\Vert {\mathbb{E}}
_{0}  (  g_{s,t}(X_{k})) - {\mathbb E} (g_{s,t} (X_k)  ) \Vert_{p}\leq
C\Big( \frac{1}{\varepsilon k^a}+ \varepsilon^{\alpha/p} + \xi^{k/p}\Big)\, .
$$
Choosing $\varepsilon=k^{-ap/(\alpha+p)}$, we obtain 
$$
\Vert {\mathbb{E}}
_{0}  (  g_{s,t}(X_{k})) - {\mathbb E} (g_{s,t} (X_k)  ) \Vert_{p}\leq
C\Big( \frac{1}{k^{a\alpha/(\alpha+p)}} + \xi^{k/p}\Big)\, ,
$$
proving the first inequality.

In the same way
$$
\Vert g_{s,t}(X_{0}) - {\mathbb{E}}
_{k}( g_{s,t}(X_{0}))\Vert_{p}
\leq \Vert g_{s,t, \varepsilon}(X_{0}) - {\mathbb{E}}
_{k}( g_{s,t, \varepsilon}(X_{0}))\Vert_{p}
+ C\varepsilon^{\alpha/p}\, .
$$
Applying (\ref{THM2a}) of  Theorem \ref{THM2}, we obtain 
$$
\Vert g_{s,t}(X_{0}) - {\mathbb{E}}
_{k}( g_{s,t}(X_{0}))\Vert_{p} \leq
C(\varepsilon^{-1}\omega_{(u)}(f,\rho_{u}^k)+ \varepsilon^{\alpha/p})\, .
$$
Since $\omega_{(u)}(f,\rho_{u}^k)\leq C k^{-a}$, the choice $\varepsilon=k^{-ap/(\alpha+p)}$
gives the second inequality.

Let $h^{(0)}(X_i)=h(X_i)-{\mathbb E}(h(X_i))$. To prove the third inequality, we have to bound up
$$\sup_{i \geq 0} \Vert{\mathbb{E}}_{0}(g^{(0)}_{s,t}(X_{i})g^{(0)}_{s,t}(X_{j+i}))\Vert
_{p/2}\quad \text{and} \quad  \sup_{0 \leq i \leq j
}\Vert{\mathbb{E}}_{0}(g^{(0)}_{s,t}(X_{j})
g^{(0)}_{s,t}(X_{j+i}))-{\mathbb{E}}(g^{(0)}_{s,t}(X_{j})
g^{(0)}_{s,t}(X_{j+i}))\Vert
_{p/2} \, .
$$
Using the decomposition (\ref{dec}), and the fact that the distribution functions
of $f_1$ and $f_2$ are H\"older continuous of order $\alpha$, we get 
\begin{equation}\label{one}
\Vert{\mathbb{E}}_{0}(g^{(0)}_{s,t}(X_{i})g^{(0)}_{s,t}(X_{j+i}))\Vert
_{p/2} \leq
\Vert{\mathbb{E}}_{0}(g^{(0)}_{s,t, \varepsilon}(X_{i})
g^{(0)}_{s,t,\varepsilon}(X_{j+i}))\Vert_{p/2} + C\varepsilon^{2\alpha/p}\, ,
\end{equation}
and
\begin{multline}\label{onebis}
\Vert{\mathbb{E}}_{0}(g^{(0)}_{s,t}(X_{j})
g^{(0)}_{s,t}(X_{j+i}))-{\mathbb{E}}(g^{(0)}_{s,t}(X_{j})
g^{(0)}_{s,t}(X_{j+i}))\Vert
_{p/2} \\ \leq
\Vert{\mathbb{E}}_{0}(g^{(0)}_{s,t, \varepsilon}(X_{j})
g^{(0)}_{s,t, \varepsilon}(X_{j+i}))-{\mathbb{E}}(g^{(0)}_{s,t, \varepsilon}(X_{j})
g^{(0)}_{s,t, \varepsilon}(X_{j+i}))\Vert
_{p/2}
+ C\varepsilon^{2\alpha/p}\, .
\end{multline}
Writing
\begin{multline}
\Vert{\mathbb{E}}_{0}(g^{(0)}_{s,t, \varepsilon}(X_{i})
g^{(0)}_{s,t,\varepsilon}(X_{j+i}))\Vert_{p/2}
\leq
\Vert{\mathbb{E}}_{0}((g_{s,t, \varepsilon}(X_{i})-
{\mathbb E}(g_{s,t, \varepsilon}(X_{i})|{\mathcal F}_{i+[j/2]}))
g^{(0)}_{s,t,\varepsilon}(X_{j+i}))\Vert_{p/2}\\
+
\Vert {\mathbb{E}}_{0}(
{\mathbb E}(g_{s,t, \varepsilon}(X_{i})|{\mathcal F}_{i+[j/2]})
g^{(0)}_{s,t,\varepsilon}(X_{j+i}))\Vert_{p/2} \, ,
\end{multline}
and arguing as in Lemma \ref{lmamajcov}, we infer that
\begin{equation}\label{oneter}
\Vert{\mathbb{E}}_{0}(g^{(0)}_{s,t, \varepsilon}(X_{i})
g^{(0)}_{s,t,\varepsilon}(X_{j+i}))\Vert_{p/2}
\leq C \Big( \frac{1}{\varepsilon j^a} + \xi^{[j/2]}\Big)\, .
\end{equation}
From (\ref{one}) and (\ref{oneter}), we obtain the bound
$$
\Vert{\mathbb{E}}_{0}(g^{(0)}_{s,t}(X_{i})g^{(0)}_{s,t}(X_{j+i}))\Vert
_{p/2} \leq C \Big( \frac{1}{\varepsilon j^a} + \varepsilon^{2\alpha/p}
 + \xi^{[j/2]}\Big)\, .
$$
Taking $\varepsilon=j^{-ap/(2\alpha+p)}$, we obtain 
\begin{equation}\label{two}
\sup_{i \geq 0}\Vert{\mathbb{E}}_{0}(g^{(0)}_{s,t}(X_{i})g^{(0)}_{s,t}(X_{j+i}))\Vert
_{p/2} \leq C j^{-2a\alpha/(2\alpha+p)}\, .
\end{equation}

Let $\varphi:=g_{s,t,\varepsilon}\circ f- \bar \lambda(g_{s,t,\varepsilon}\circ f) $.
Applying   Theorem \ref{THM2} (see (\ref{THM2d})), for $i\leq j$,
\begin{multline*}
\Vert{\mathbb E}_0(g^{(0)}_{s,t, \varepsilon}(X_{j})
g^{(0)}_{s,t, \varepsilon}(X_{j+i}))-{\mathbb{E}}(g^{(0)}_{s,t, \varepsilon}(X_{j})
g^{(0)}_{s,t, \varepsilon}(X_{j+i}))\Vert_{p/2}
= \Vert{\mathbb E}(\varphi.\varphi\circ T^i|\mathcal F_{-j}) -
     {\mathbb E}(\varphi .\varphi\circ T^i)\Vert_{p/2}\\
\leq C(\xi^{2j/p}+\omega_{(s,e)}(\varphi.\varphi\circ T^i,\zeta^j))\, .
\end{multline*}
By (\ref{hyp-b}),
$
   \omega_{(s,e)}(\varphi.\varphi\circ T^i,\zeta^j) \leq
     2 \|\varphi\|_\infty \omega_{(s,e)}(\varphi,K\zeta^j j^{d_e}))
     \leq
     4  \omega_{(s,e)}(\varphi,L\zeta_0^j)
$ where $\zeta_0 \in (\zeta, 1)$. Hence, 
\begin{equation}\label{ouf}
\Vert{\mathbb E}_0(g^{(0)}_{s,t, \varepsilon}(X_{j})
g^{(0)}_{s,t, \varepsilon}(X_{j+i}))-{\mathbb{E}}(g^{(0)}_{s,t, \varepsilon}(X_{j})
g^{(0)}_{s,t, \varepsilon}(X_{j+i}))\Vert_{p/2} \leq
C(\xi^{2j/p}+\omega_{(s,e)}(\varphi, L\zeta_0^j))\, .
\end{equation}
Since $\omega_{(s,e)}(\varphi, L\zeta_0^j)\leq \varepsilon^{-1}
\omega_{(s,e)}(f,L\zeta_0^j)\leq
C\varepsilon^{-1}j^{-a}$, we obtain from (\ref{onebis}) and (\ref{ouf}) that
$$
\Vert{\mathbb{E}}_{0}(g^{(0)}_{s,t}(X_{j})
g^{(0)}_{s,t}(X_{j+i}))-{\mathbb{E}}(g^{(0)}_{s,t}(X_{j})
g^{(0)}_{s,t}(X_{j+i}))\Vert
_{p/2}
\leq C\Big(\frac{1}{\varepsilon j^{a}}+\varepsilon^{2\alpha/p}+\xi^{2j/p}\Big) \, .
$$
Taking $\varepsilon=j^{-ap/(2\alpha+p)}$, we obtain 
\begin{equation}\label{three}
\sup_{0\leq i \leq j}
\Vert{\mathbb{E}}_{0}(g^{(0)}_{s,t}(X_{j})
g^{(0)}_{s,t}(X_{j+i}))-{\mathbb{E}}(g^{(0)}_{s,t}(X_{j})
g^{(0)}_{s,t}(X_{j+i}))\Vert
_{p/2} \leq C j^{-2a\alpha/(2\alpha+p)}\, .
\end{equation}
The third inequality of Lemma \ref{lmacalculcoeff}
follows from (\ref{two}), (\ref{three}) and from the definition of
the quantity  $A(g_{s,t}(X)-{\mathbb E}(g_{s,t}(X)),j)$
given in Proposition \ref{consdirect}.
\end{proof}

\section{Additional results for partial sums} \label{adresults}

Let $T$ be an ergodic automorphism of ${\mathbb T}^d$ as defined in the introduction. Let $f$ be a continuous function from  ${\mathbb T}^d$ to ${\mathbb R}$
with modulus of continuity $\omega(f, \cdot)$.

The inequalities given in Theorem \ref{THM2} have been used to prove the
tightness of the sequential empirical process, but they
can be used in many other situations. Let us give three examples of application
to the behavior of the partial sums (\ref{partial}).

\begin{enumerate}
\item {\bf Moment bounds for partial sums}.   Using Corollary \ref{corburk} together with Theorem \ref{THM2} (see also Remark \ref{remarqueTH2}), we infer that if 
\begin{equation}\label{firstcond}
 \sum_{n>0} \frac{\omega(f,\zeta^n)}{\sqrt n} < \infty\, ,
\end{equation}
where $\zeta \in (0,1)$ is defined in Theorem \ref{THM2},
then for any $p>2$, $$
\Big \Vert  \max_{1 \leq k \leq n}  
\Big | \sum_{i=1}^k  (f \circ T^i - \bar{\lambda} (f)) \Big |  \Big \Vert_p \ll n^{1/2} \, .
$$
Clearly, the condition (\ref{firstcond}) is equivalent to the integral condition 
\begin{equation}\label{firstcondbis}
 \int_0^{1/2} \frac{\omega(f,t)}{t |\log t|^{1/2}} dt < \infty \, . 
\end{equation}
\item {\bf Weak invariance principle}. If the integral condition (\ref{firstcondbis})
holds
then the series
\begin{equation}\label{covseries}
\sigma^2(f)= \bar \lambda ((f- \bar \lambda (f))^2) + 2 \sum_{k>0} \bar \lambda ((f-\bar \lambda (f)) \cdot f \circ T^k)
\end{equation}
converges absolutely, and  the process
$$
\Big \{ \frac{1}{\sqrt n} \sum_{k=1}^{[nt]} (f \circ T^k - \bar{\lambda} (f))
,  t \in [0,1] \Big \}
$$ converges to a Wiener process with variance $\sigma^2(f)$  in the space $D([0,1])$ of c\`adl\`ag function equipped with the uniform metric. This follows from Theorem  \ref{THM2}
together with  Proposition 5 in \cite{DeMeVo}.

\item {\bf Rates of convergence in the strong invariance principle}. Let $p \in ]2,4]$, and assume that   $$\omega(f, x)\leq C |\log (x)|^{-a}  \ \text{in a neighborhood
of $0$ for  some} \  a>\frac{1+\sqrt{1+4p(p-2)}}{2p}+1-\frac{2}{p}\, .$$ Then,  enlarging ${\mathbb T}^d$ if necessary, there exists a sequence 
$(Z_i)_{ i \geq 1}$ of independent and identically distributed Gaussian random variables with mean zero and variance $\sigma^2(f)$ defined in (\ref{covseries}) such that, for any $t>2/p$,
$$
\sup_{1 \leq k \leq n} \Big |\sum_{i=1}^k (f \circ T^i- \bar{\lambda} (f)) - \sum_{i=1}^k Z_i \Big | = o\big(n^{1/p} (\log (n))^{(t+1)/2}\big)\quad \text{almost surely as $n \rightarrow \infty$.}
$$
In particular, we obtain the rate of convergence $ n^{1/2- \epsilon}$ for some $\epsilon>0$ as soon as $a>1/2$, and the rate $n^{1/4} \log (n)$ as soon as $a\geq 3/2$. 
This follows from Theorem \ref{THM2} together with Theorem 3.1 in \cite{DeMePene1}.
\end{enumerate}

\section{Appendix}

\setcounter{equation}{0}

In this section, we prove Remark \ref{cardan}, so we give the solutions of the equation \eqref{debileq}. We first write (\ref{debileq}) under the following form $p^3+bp^2+cp+d=0$. Following the classical Cardan method, we set $p':=-\frac{b^2}3+c$ and $q:=\frac b{27}(2b^2-9c)+d$ (this leads to the formulas for $p'$ and $q$  as given in Remark \ref{cardan}). Observe that $p^3+bp^2+cp+d=0$ means that $z=p+\frac b3$ satisfies $z^3+p'z+q=0$.
We then compute as usual $\Delta:=q^2+\frac 4{27}(p')^3$. We get
\begin{multline*}\Delta = ((64/27)\ell-(64/27)\ell^2-16/27)\alpha^4+(-(128/27)\ell^3 \\ +(128/27)\ell^2-(32/9)\ell)
\alpha^3+((32/27)\ell-(64/27)\ell^4+(16/27)\ell^2
-16/27-(128/27)\ell^3)\alpha^2 \\+
(-(32/9)\ell-(32/27)\ell^2-(64/27)\ell^4-(32/9)\ell^3)\alpha-(16/27)\ell^2-(16/27)\ell^4<0 \, .
\end{multline*}
Since $\Delta$ is negative, we use the usual expression of the solutions $z$ with $\cos$ and $\arccos$ 
(to which we substract $b/3$). So the solutions are  $$p_k=2\frac{\ell+1-\alpha}3+2\sqrt{-\frac {p'} 3}\cos\left(\frac 13\arccos\left(-\frac q2\sqrt{\frac {27}{-(p')^3}}\right)+\frac{2k\pi}3\right) $$
for $k\in\{0,1,2\}$. Clearly $p_1 < p_2 < p_0$. The unique solution in $]2 \ell , 4 \ell [$ is then $p_0$.

\bigskip

\noindent {\bf Acknowledgements.} The authors would like to thank the two referees for
carefully reading the manuscript and for numerous suggestions that improved the presentation of this paper.

\end{document}